\documentclass[a4paper,11pt]{article}
\usepackage[normalem]{ulem}
\usepackage[margin=0.8in]{geometry}
\usepackage{hyperref}
\usepackage[utf8]{inputenc}
\usepackage[maxbibnames=99, sortcites]{biblatex}
\usepackage{framed}
\usepackage{graphicx}
\usepackage{amsfonts}
\usepackage{amsmath}
\usepackage{amsthm}
\usepackage{amssymb}
\usepackage{mathrsfs}

\usepackage{color}
\usepackage[dvipsnames,table,xcdraw]{xcolor}
\definecolor{darkred}{rgb}{0.8, 0.0, 0.0}
\usepackage{hyperref}
\hypersetup{
    colorlinks=true,
    urlcolor = {darkred},
    linkcolor = {darkred},
    citecolor = {darkred},
    linkcolor = {darkred},
    linktoc=all
}
\usepackage{kotex}

\colorlet{Color0}{Dandelion}
\colorlet{Color1}{red}
\colorlet{Color2}{Green}
\colorlet{Color3}{BlueViolet}
\usepackage{kotex}
\usepackage{hyperref}
\usepackage{authblk}
\usepackage{multirow}
\usepackage{hhline}
\usepackage{subcaption}
\usepackage{color}
\usepackage{cleveref}
\usepackage[shortlabels]{enumitem}
\usepackage{bbm}
\usepackage{tikz-cd}

\usepackage{lineno}

\definecolor{darkblue}{rgb}{0.0, 0.0, 0.8}
\definecolor{darkred}{rgb}{0.8, 0.0, 0.0}
\definecolor{darkgreen}{rgb}{0.3, 0.5, 0.15}

\usepackage{hyperref}
\hypersetup{
    colorlinks=true,
    urlcolor = {darkred},
    linkcolor = {darkred},
    citecolor = {darkred},
    linkcolor = {darkred},
    linktoc=all
}

\newcommand{\R}{\mathbb{R}}
\newcommand{\K}{\mathbb{K}}
\newcommand{\M}{\mathbb{M}}
\newcommand{\vol}[1]{\mathrm{vol}(#1)}
\newcommand{\ellK}[1]{\ell_{#1}}
\newcommand{\tK}[1]{t_{#1}}

\newtheorem{theorem}{Theorem}[section]
\newtheorem{corollary}[theorem]{Corollary}
\newtheorem{proposition}[theorem]{Proposition}
\newtheorem{lemma}[theorem]{Lemma}

\theoremstyle{definition}
\newtheorem{definition}[theorem]{Definition}
\newtheorem{remark}[theorem]{Remark}
\newtheorem{example}[theorem]{Example}
\newtheorem{convention}[theorem]{Convention}
\newtheorem{notation}[theorem]{Notation}

\title{Limit Theorems for Verbose Persistence Diagrams}

\author[1]{Jeong-hwi Joe}
\author[2]{Woojin Kim}
\author[3]{Cheolwoo Park}

\affil[1,2,3]{Department of Mathematical Sciences, KAIST, South Korea\thanks{\texttt{jhjoe@kaist.ac.kr},  
\texttt{woojin.kim@kaist.ac.kr}, 
\texttt{parkcw2021@kaist.ac.kr}}}

\date{\today}

\begin{document}


\allowdisplaybreaks

\maketitle

\begin{abstract}
    The persistence diagram is a central object in the study of persistent homology and has also been investigated in the context of random topology. The more recent notion of the \emph{verbose diagram} (a.k.a. verbose barcode) is a refinement of the persistence diagram. 
    Whereas the persistence diagram is a complete invariant of persistent homology, the verbose persistence diagram is a complete invariant of the one-level higher object---a filtered chain complex. It therefore strictly contains the persistence diagram, both in form and in the amount of information it encodes. Concretely, the verbose diagram incorporates ephemeral features that arise in a filtered topological space, representing them as additional points along the diagonal. 

In this work, we initiate the study of \emph{random} verbose diagrams. We establish a strong law of large numbers for verbose diagrams as a random point cloud grows in size---that is, we prove the existence of a limiting verbose diagram, viewed as a measure on the half-plane on and above the diagonal. Also, we characterize its support and compute its total mass.  Along the way, we extend the notion of the persistent Betti number,    reveal the relation between this extended notion and the verbose diagram (which is an extension of the fundamental lemma of persistent homology), and establish results on the asymptotic behavior of the extended persistent Betti numbers.
    
    This work extends the main results of the work by Hiraoka, Shirai, and Trinh and its sequel by Shirai and Suzaki to the setting of verbose diagrams.
\end{abstract}
\bigskip
\noindent\textbf{2020 Mathematics Subject Classification.}  60K35, 60B10, 55N31.

\section{Introduction}

This work bridges two lines of research, which are reviewed below.

\paragraph{Random topology and persistence diagrams.}

The first line of research stems from \emph{random topology}, which studies the topological features of randomly generated spaces.
A common choice for such spaces is a simplicial complex, with various models of generation (see \cite{bobrowskietal2022}).
Among them, random geometric complexes~\cite{kahle2011} extend the idea of random geometric graphs~\cite{penrose2003} through the following typical construction: fix a threshold $r \geq 0$ and start with a random subset $\Xi$ of a Euclidean space. Then, consider the simplicial complex consisting of subsets $\sigma \subseteq \Xi$ whose ``size'' is at most $r$, where ``size'' may be defined in different ways, for example, as the diameter of $\sigma$ (giving the \emph{Vietoris--Rips complex}) or as the radius of its smallest enclosing ball (giving the \emph{\v{C}ech complex}).

Now, recall that \emph{homology} from algebraic topology captures ``holes" of different dimensions in a simplicial complex, such as connected components (0-dim.), loops (1-dim.), cavities (2-dim.), and their high-dimensional analogues.
Accordingly, the topology of a random geometric complex can be studied via its homology, and the number of $q$-dimensional holes---known as the $q$th Betti number $\beta_q$---becomes a random variable.
Among other directions, analyzing the asymptotic behavior of Betti numbers as the 
cardinality of $\Xi$ grows has been an active area of research \cite{bobrowski2017, bobrowski2022, owada2017, yogeshwaran2017}.

When homology is applied to the analysis of discrete data, \emph{persistent homology} arises; it is a central tool in topological data analysis (TDA), with numerous applications~\cite{carlsson2009topology, carlsson2021topological, dey2022computational}.
In persistent homology, one typically starts with data given as a \emph{point cloud} $\Xi$ (i.e. a finite set of points in a Euclidean space) and builds a family $\K(\Xi)=\{K_t\}_{t\in [0,\infty)}$ of simplicial complexes such that for $r\leq s$ in $[0,\infty)$, we have $K_r\subset K_s$. This $\K(\Xi)$ is called a \emph{(simplicial) filtration}. For each $q=0,1,\ldots$, the $q$-th homological features of $\K(\Xi)$ is summarized into the $q$-th \emph{persistence diagram}, which records the births and deaths of $q$-dimensional holes in $\K(\Xi)$ (see Figure~\ref{fig: filtration-diagram}).  
Also, 
the $q$-th Betti number $\beta_q$ is refined to the \emph{persistent Betti number} $\beta_q^{r, s}$ for $r\leq s$ $[0,\infty)$, 
which counts the number of $q$-dimensional holes that persist from $t = r$ to $t = s$ and can be read from the persistence diagram (see Figure~\ref{subfig: PBN}).

\begin{figure}
    \centering
    \includegraphics[width=0.6\linewidth]{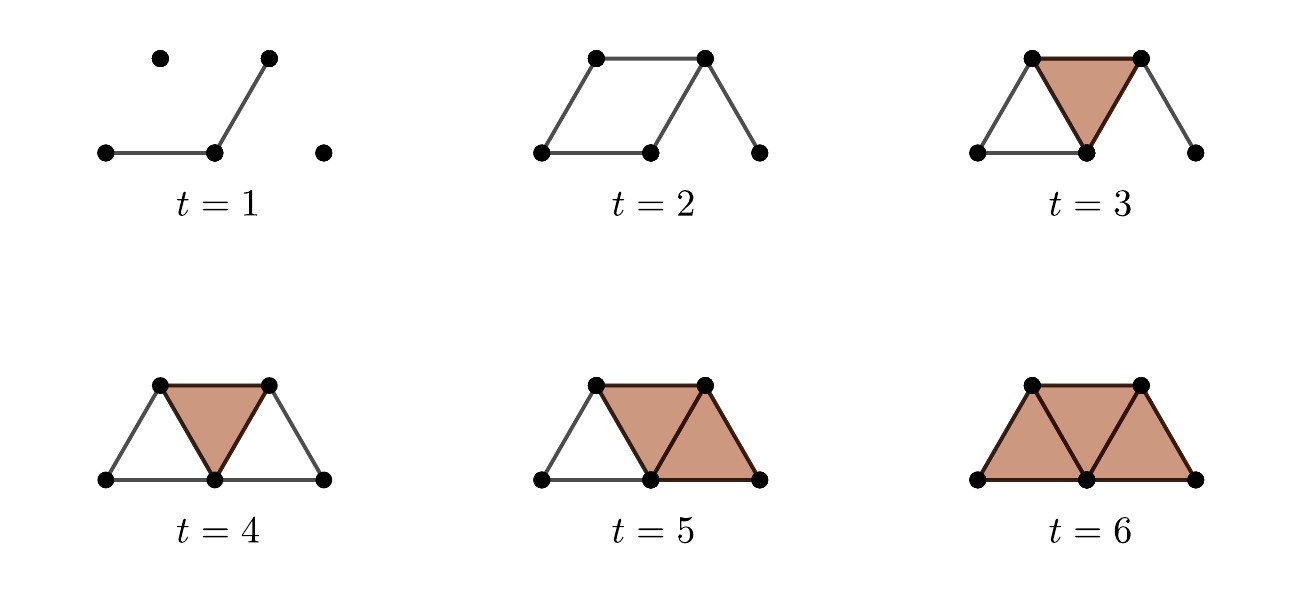}
    \includegraphics[width=0.3\linewidth]{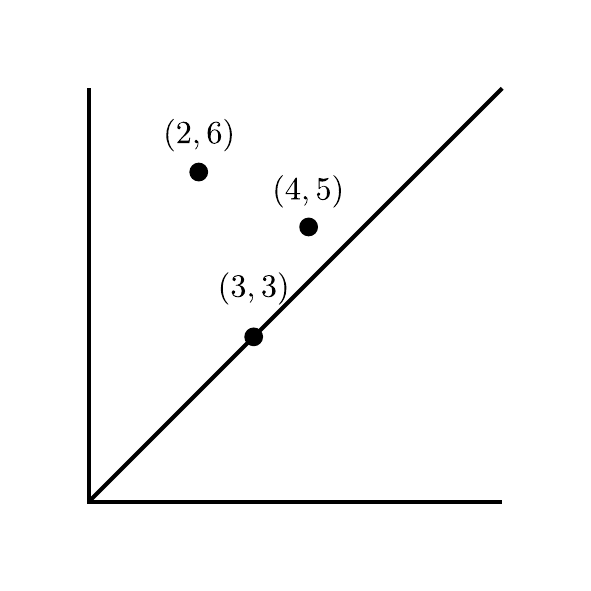}
    \caption{
 Assume that the simplicial filtration $\K = \K(\Xi)$ is as depicted on the left above, 
being constant on the intervals $[n, n+1)$ for $n = 1, \ldots, 5$.
        A loop is created at \(t = 2\) and persists until it is filled in at \(t = 6\), and another loop is created at \(t = 4\) and is filled in at \(t = 5\). Thus, the points \((2, 6)\) and \((4, 5)\) constitute the degree-\(1\) persistence diagram of $\K$ depicted on the right. The degree-1 verbose diagram of $\K$ is obtained by adding the point $(3,3)$ to the persistence diagram, reflecting the fact that a cycle is born and dies simultaneously at $t = 3$.}
    \label{fig: filtration-diagram}
\end{figure}

Note that, given a \emph{random} subset $\Xi$ of a Euclidean space, the simplicial filtration $\K(\Xi)$ is also random. For example, such a random $\Xi$ may arise from a point process in Euclidean space, modeling an atomic configuration of glass materials in disordered states \cite{hiraoka2016hierarchical}. In this context, understanding of the persistence diagram of  $\K(\Xi)$ can be useful for characterizing the geometry and topology of glass materials. Motivated in part by this,
the asymptotic behavior of persistence diagrams and persistent Betti numbers
arising from such a random $\Xi$
was analyzed by Hiraoka, Shirai, and Trinh~\cite{hiraokaetal2018}.
It has since inspired many subsequent studies.
Key results include limit theorems for persistence diagrams \cite{owada2020, shirai2022, owada2022, kanazawa2024large}, for persistent Betti numbers \cite{krebs2025, krebs2021}, and for multiparameter persistent Betti numbers \cite{botnan2024}. 
We also remark that \emph{universality} in random persistent homology has recently emerged as an active area of research~\cite{bobrowski2023, bobrowski2024universality}.

\paragraph{Verbose diagrams.}
The second line of research concerns the \emph{verbose diagram} (a.k.a. verbose barcode), introduced by Usher and Zhang. 
It is a refinement of the persistence diagram that is obtained by incorporating ephemeral persistence features as extra points along the diagonal (see Figure~\ref{fig: filtration-diagram}). More precisely, the verbose diagram is an invariant of a filtered chain complex (as opposed to the homology level) and completely determines its chain isomorphism type \cite[Theorem A]{usherzhang2016}. They also introduced the notion of \emph{concise diagram}, which turns out to be equivalent to that of the persistence diagram in a suitable sense (see Section~\ref{subsec: verbose}). The verbose and concise diagrams were employed in the study of Hamiltonian Floer theory~\cite[Section 12]{usherzhang2016}.

Belton et al. and Fasy et al. employed the verbose diagram in the study of the persistent homology transform~\cite{turner2014}, showing that verbose diagrams can describe a shape faithfully using fewer filtration directions than persistence diagrams
\cite{fasy2024, belton2020}.\footnote{We remark that verbose diagrams are sometimes called augmented persistence diagrams in the literature.}
M\'{e}moli and Zhou 
showed that the verbose diagram for the Vietoris-Rips filtration of a metric space is more discriminative than
the persistence diagram \cite{mémolizhou2024}. 
{Chacholski et al. 
studied invariants of tame parametrized chain complexes, a generalization of filtered chain complexes allowing non-injective maps \cite{chacholski2021invariants}. In the finite-dimensional case, their \emph{Betti} and \emph{minimal Betti} diagrams reduce to the verbose and concise barcodes in \cite{usherzhang2016}. Later, an algorithm to decompose filtered chain complexes into indecomposables was introduced \cite{chacholski2023decomposing}.} We remark that when a filtered chain complex arises from a simplicial filtration, its verbose barcode can also
be obtained through the usual matrix reduction procedure applied to the boundary matrix of
the underlying simplicial filtration \cite[Ch. VII]{edelsbrunner2010}.

When it comes to interpreting persistence diagrams, a prevalent perspective has been to regard points far from the diagonal as features arising from significant topological structures in the input data, whereas points near the diagonal are often attributed to noise.
However, this view is no longer universally accepted. For example, \cite{bubenik2020persistent} shows that the ensemble of points near the diagonal captures information about the curvature of the underlying manifold from which the point cloud is sampled. Moreover, the aforementioned enhanced discriminative power of the verbose diagram suggests that 
useful geometric information may also be encoded in its diagonal points.
 
In sum, the verbose diagram is a subject worth further investigation, and this work advances the understanding of its asymptotic behavior as the input point cloud data grows.

\paragraph{Known results.} We unify the two threads of research discussed above and initiate the study of \emph{random} verbose diagrams. In order to describe our contributions, we first summarize the main results of \cite{hiraokaetal2018} and \cite{shirai2022}, which we will extend to the setting of verbose diagrams.

\begin{itemize}
    \item (Strong law of large numbers for persistence diagrams)
    To model the growth of a random point cloud over time, let us consider a stationary point process on \(\R^N\),
    restricted to a growing observation window.
    The persistence diagram of the restricted point cloud, viewed as a random measure in the form of a sum of Dirac measures on $\Delta$ (Figure~\ref{subfig: delta}), 
    converges to a limiting measure (up to normalization)---which is deterministic---as the window grows 
    (Theorem~\ref{thm: vagueconv}).
    
    \item (Support of the limiting persistence diagram)
    The support of the limiting measure is equal to the closure of the set of \emph{realizable points}---those that can appear in the persistence diagram of some point cloud---under mild conditions on the point process and the filtration construction 
    (Theorem~\ref{thm: measuresupport}).

    \item (Strong law of large numbers for the persistent Betti number)
    The persistent Betti numbers of the filtration built on the restricted point cloud are asymptotically proportional to the volume of the window (Theorem~\ref{thm: persbettiSLLN}). 

    \item (Central limit theorem for the persistent Betti number)  
  Persistent Betti numbers for a homogeneous Poisson point process are asymptotically normal (Theorem~\ref{thm: persbettiCLT}).
    \footnote{
    This result extends the central limit theorem for Betti numbers~\cite{yogeshwaran2017} to the persistent setting.}    
\end{itemize}

In particular, Theorems~\ref{thm: vagueconv}~and~\ref{thm: persbettiSLLN} are extended to the marked point process setting (Remark~\ref{rmk: shirai}).

\paragraph{Our contributions.}

We extend the full set of main theorems from \cite{hiraokaetal2018} and \cite{shirai2022} summarized above, to the setting of verbose diagrams, as summarized
in Table~\ref{tab: summaryofresults}. In particular, our Theorem~\ref{thm: vagueconvverbose} implies that the extra points added to the persistence diagram along the diagonal exhibit asymptotic regularity just like the off-diagonal points, as established in earlier works. We also extend the definition of persistent Betti number \(\beta_q^{r, s}\) to allow the case \(r > s\), which can be read from the verbose diagram (see Figure~\ref{subfig: PBNextended}).
For a more detailed interpretation of the extended persistent Betti number, see Remark~\ref{rmk: PBNextended}.
\begin{table}
    \centering
        \begin{tabular}{l|l|l|}
            \hhline{~|-|-|}
             & \cellcolor[HTML]{C0C0C0}PD \& PBN & \multicolumn{1}{|l|}{\cellcolor[HTML]{C0C0C0}\begin{tabular}[c]{@{}l@{}}VD \& EPBN\\ \textbf{(our contributions)}\end{tabular}}\\ \hline
            \multicolumn{1}{|l|}{\cellcolor[HTML]{C0C0C0}\begin{tabular}[c]{@{}l@{}}Strong law of large numbers for the PD and VPD\\ (i.e. existence of limiting PD and VD) 
            \end{tabular}} & \begin{tabular}[c]{@{}l@{}}Theorem~\ref{thm: vagueconv} \\ (Remark~\ref{rmk: shirai})\end{tabular}  & Theorem~\ref{thm: vagueconvverbose} \\ \hline
            \multicolumn{1}{|l|}{\cellcolor[HTML]{C0C0C0}\begin{tabular}[c]{@{}l@{}}Total mass of the limiting PD and VD 
            \end{tabular}} &  & Theorem~\ref{thm: limitmeasuretotalmass} \\ \hline
            \multicolumn{1}{|l|}{\cellcolor[HTML]{C0C0C0}\begin{tabular}[c]{@{}l@{}}Support of the limiting PD and VD 
            \end{tabular}} & Theorem~\ref{thm: measuresupport}  & Theorem~\ref{thm: measuresupportverbose} \\ \hline
            \multicolumn{1}{|l|}{\cellcolor[HTML]{C0C0C0}\begin{tabular}[c]{@{}l@{}}Strong law of large numbers for the PBN and EPBN
            \end{tabular}} & \begin{tabular}[c]{@{}l@{}}Theorem~\ref{thm: persbettiSLLN} \\ (Remark~\ref{rmk: shirai})\end{tabular}  & Theorem~\ref{thm: persbettiSLLNextended} \\ \hline
            \multicolumn{1}{|l|}{\cellcolor[HTML]{C0C0C0}\begin{tabular}[c]{@{}l@{}}Central limit theorem for the PBN and EPBN
            \end{tabular}} & Theorem~\ref{thm: persbettiCLT}  & Theorem~\ref{thm: persbettiCLTextended} \\ \hline
        \end{tabular}
    \caption{
        In this work, we extend the limit theorems for persistence diagrams (PDs) and persistent Betti numbers (PBNs) presented in \cite{hiraokaetal2018,shirai2022} to verbose diagrams (VDs) and extended persistent Betti numbers (EPBNs).
    }
    \label{tab: summaryofresults}
\end{table}

\begin{figure}
    \centering
    \begin{subfigure}{0.27\linewidth}
        \includegraphics[width=\textwidth]{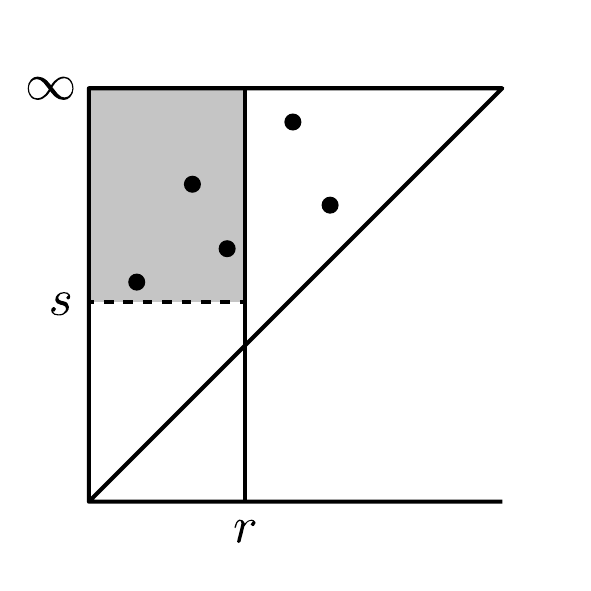}
        \caption{}
        \label{subfig: PBN}
    \end{subfigure}
    \begin{subfigure}{0.27\linewidth}
        \includegraphics[width=\textwidth]{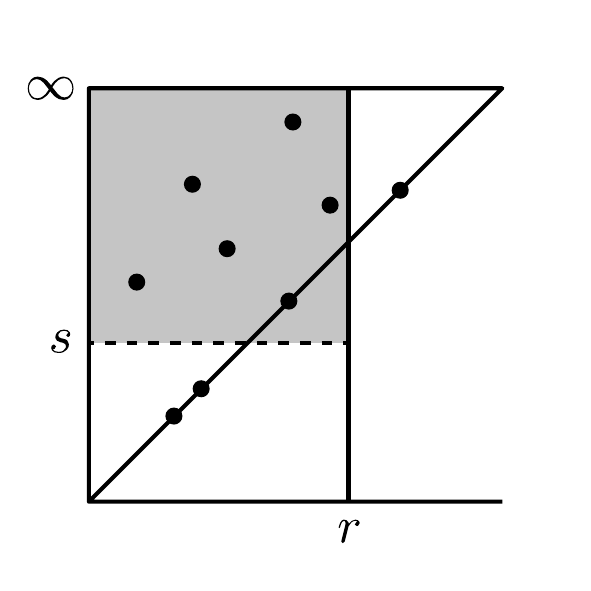}
        \caption{}
        \label{subfig: PBNextended}
    \end{subfigure}
    \caption{
        \subref{subfig: PBN} 
        For $r\leq s$ in $\R_{\geq 0}$, the persistent Betti number \(\beta_q^{r, s}\) (defined in Equation~\eqref{eq: persbetti}) equals the number of points in the persistence diagram that fall into the shaded region $[0,r]\times (s,\infty]$. This equality is known as the \emph{fundamental lemma of persistent homology} \cite[Chapter VII]{edelsbrunner2010}. 
        \\ \subref{subfig: PBNextended} 
        In this work, we extend \(\beta_q^{r, s}\) in a natural way so that it is well defined even for \(r > s\), and establish the fundamental lemma of persistent homology for the \emph{extended} persistent Betti number and the verbose diagram. Namely, we show that \(\beta_q^{r, s}\) equals the number of points in the verbose diagram that fall into the region $[0,r]\times (s,\infty]$ as before. We also show that this extended persistent Betti number satisfies a central limit theorem (see Definition~\ref{def: extendedPBN}, Proposition \ref{prop: persbettimeasure}, and Theorem~\ref{thm: persbettiCLTextended}).
    }
    \label{fig: persistent-betti-number-measure}
\end{figure}

\paragraph{Main ideas behind the proofs.}
To prove Theorem~\ref{thm: vagueconvverbose}, we utilize the idea we call “shifting the diagram.”
Namely, we modify the construction of a filtration on a point cloud, so that the resulting verbose diagram is obtained by shifting the original verbose diagram upward along the $y$-axis.
This makes every point of the new diagram lie outside the diagonal, enabling us to see the resulting verbose diagram as the ordinary persistence diagram.
We then apply the prior theorem (Theorem~\ref{thm: vagueconv}) and argue that the limiting measure for the ordinary persistence diagram of the modified filtration is a translation of the limiting measure of the original verbose diagram.
Due to the presence of the infinity line \([0, \infty) \times \{\infty\}\) on the extended plane, the notion of translation becomes technically subtle, and much of the proof is devoted to handling this subtlety.  
Theorems \ref{thm: persbettiSLLNextended} and \ref{thm: persbettiCLTextended} are proved using a similar idea of shifting the diagram, combined with our Proposition~\ref{prop: cycleboundary}, which relates diagram points to basis elements of the cycle and boundary groups.

In Theorem~\ref{thm: limitmeasuretotalmass}, we identify the total mass of the limiting measure from Theorem~\ref{thm: vagueconvverbose}.
For this, we utilize a lemma
that, under a mild assumption, the number of points in the verbose diagram is determined by the number of points in the underlying point cloud (Lemma~\ref{lem:verbose-cardinality}).  
We also use Jensen’s inequality, continuity of measures, basic linear algebra, and Proposition~\ref{prop: cycleboundary} mentioned above.

Finally, our proof of Theorem~\ref{thm: measuresupportverbose} parallels the proof of Theorem~\ref{thm: measuresupport} from \cite{hiraokaetal2018}. The main ingredient in establishing Theorem~\ref{thm: measuresupportverbose} is Lemma~\ref{lem: matchingdistancecontinuity}, which shows that if the input point cloud is perturbed slightly, then each point in the resulting verbose diagram moves by a correspondingly small amount. 

\paragraph{Organization.} 

Section~\ref{sec: prelim} provides preliminaries necessary for the remainder of this paper.
Section~\ref{sec: prevresults} summarizes the main results from \cite{hiraokaetal2018} and \cite{shirai2022}.
Section~\ref{sec: mainresults} presents our main results, which extend the results from the previous section.

\paragraph{Acknowledgement.}
JJ and WK thank 
Khanh Duy Trinh for helpful correspondence. This work was partially supported by the National Research Foundation of Korea (NRF) grant funded by the Korea government (MSIT) (RS-2025-00515946 and 2021R1A2C1092925).

\section{Preliminaries} \label{sec: prelim}

Section~\ref{subsec: measure} reviews the basics of random measures and some key concepts of (marked) point processes.
Section~\ref{subsec: pershom} reviews the basics of persistent homology.
Section~\ref{subsec: verbose} reviews the notion of the verbose diagram.

\subsection{Random Measures and Point Processes} \label{subsec: measure}

We review the basics of random measures and (marked) point processes.
See \cite{kallenberg2017} for a detailed introduction.

Suppose that \(S\) is a second-countable locally compact Hausdorff space.  
Then it is well known that \(S\) is Polish, i.e., \(S\) is a separable and completely metrizable space \cite[Theorem~5.3]{kechris2012classical}.
A measure \(m\) on the Borel \(\sigma\)-algebra of \(S\) is called a \emph{Radon measure} (or \emph{locally finite measure}) if \(m(F) < \infty\) for every compact subset \(F\) of \(S\).  
Let \(\mathscr{R}(S)\) denote the set of all Radon measures on \(S\).  
A sequence \(\{\mu_n\}_n\) in \(\mathscr{R}(S)\) is said to \emph{converge vaguely} to \(\mu \in \mathscr{R}(S)\), and we write \(\mu_n \xrightarrow{v} \mu\), if
\[
    \int_S f \, d\mu_n \to \int_S f \, d\mu \quad \text{as } n \to \infty,
\]
for every continuous function \(f: S \to \R\) with compact support.
For any subset \(B \subset S\), let \(B^\circ\) and \(\overline{B}\) denote the interior and closure of \(B\), respectively.
Recall that $B$ is called \emph{relatively compact} if its closure \(\overline{B}\) is compact.

Vague convergence can be characterized as follows.

\begin{lemma}[{\cite[Lemma~4.1]{kallenberg2017}}] \label{lem: vagueconvequiv}
    Let \(\mu, \mu_1, \mu_2, \ldots \in \mathscr{R}(S)\). 
    Then the following are equivalent:
    \begin{enumerate}[(i)]
        \item \(\mu_n \xrightarrow{v} \mu\);
        \item For every relatively compact Borel set \(B \subset S\),
        \[
            \mu(B^\circ) \leq \liminf_{n \to \infty} \mu_n(B) \leq \limsup_{n \to \infty} \mu_n(B) \leq \mu\left(\overline{B}\right).
        \]
    \end{enumerate}
\end{lemma}

In this paper, we work with locally finite random point sets.  
We identify them with sums of Dirac measures, which allows us to utilize random measure theory.
In fact, such measures are precisely the integer-valued Radon measures:

\begin{proposition} \label{prop: intvaluedirac}
    Any integer-valued \(\mu \in \mathscr{R}(S)\) equals a finite or countably infinite sum of Dirac measures \(\delta_{x}\) at points \(x \in S\).
\end{proposition}

Although this proposition is well-known, we include its proof in Appendix~\ref{apdx: pfofintvaluedirac} for completeness.

Let \(\mathscr{N}(S)\) denote the subset of \(\mathscr{R}(S)\) consisting of all integer-valued measures.  
By the preceding proposition, any element of \(\mathscr{N}(S)\) can be expressed as
\(\sum_{i \in I} \beta_i \delta_{\sigma_i}\),  $\beta_i\in \mathbb{Z}_{>0},\ \sigma_i\in S$ for some index set $I$. This sum is identified with  
the multiset of points \(\sigma_i\) in \(S\), where each \(\sigma_i\) has multiplicity \(\beta_i\).

Let us equip \(\mathscr{R}(S)\) with the \emph{vague topology}, which is the initial topology generated by the maps \(\pi_f: \mu \mapsto \int_S f\,d\mu\), for all continuous functions \(f: S \to \R\) with compact support.
The vague convergence is equivalent to the convergence in vague topology \cite[Lemmas~4.3~and~4.6]{kallenberg2017}.
With the Borel \(\sigma\)-algebra induced by the vague topology, both \(\mathscr{R}(S)\) and its subset \(\mathscr{N}(S)\) become measurable spaces.
A \(\mathscr{R}(S)\)-valued random variable is called a \emph{random measure} on \(S\).  
In particular, a \(\mathscr{N}(S)\)-valued random variable is called a \emph{point process} on \(S\).
Given a random measure \(\xi\) on \(S\), the \emph{mean measure} (or \emph{intensity measure}) \(\mathbb{E}[\xi]\) is defined for all Borel sets \(A \subset S\) by
\begin{equation}\label{eq:expectation}
    \mathbb{E}[\xi](A) := \mathbb{E}[\xi(A)].
\end{equation}
A point process \(\xi\) on \(S\) is called \emph{simple} if
\[
    \mathbb{P}(\xi(\{x\}) \leq 1 \text{ for all } x \in S) = 1.
\]
A point process \(\Phi\) on \(\R^N\) is said to \emph{have all finite moments} if \(\mathbb{E}[\Phi(A)^k] < \infty\) for every bounded Borel set \(A\) and every \(k\in \mathbb{Z}_{>0}\).

The notion of a point process is generalized so that each point has a \emph{mark}.  
Examples of such marks include the temperature or altitude at a point (viewed as a location), and the color of a point.
Let \(\M\) be a second-countable locally compact Hausdorff space, representing the set of marks.
A point process \(\Phi\) on \(S \times \M\) is called a \emph{marked point process} on \(S\) (with marks in \(\M\)) if its marginal process \(\Phi_g\) on \(S\) (called the \emph{ground process}) given by \(A \mapsto \Phi(A \times \M)\) is simple.
\begin{framed}
     Throughout this paper,
    \begin{itemize}
        \item \(\M\) denotes a second-countable locally compact Hausdorff space, representing the set of marks;
        \item point processes on $\R^N$ are assumed to be simple and have all finite moments, as in \cite{hiraokaetal2018,shirai2022};
        \item in the same vein, whenever marked point processes on \(\R^N\) are considered, we assume that their ground processes have all finite moments, as in \cite{shirai2022}.
    \end{itemize}    
\end{framed}

For any measure \(\mu\) on \(\R^N \times \M\) and any vector \(v \in \R^N\), denote by \(\mu^{(v)}\) the measure defined by \(\mu^{(v)}(A) := \mu(A + v)\), where \(A + v := \{(x + v, m): (x, m) \in A\}\).
A marked point process \(\Phi\) on \(\R^N\) is called \emph{stationary} if it is translation-invariant, i.e., \(\Phi\) and \(\Phi^{(v)}\) have the same distribution for all \(v \in \R^N\).
The \emph{intensity} of a stationary marked point process \(\Phi\) on \(\R^N\) with marks in \(\M\) is the expectation \(\mathbb{E}[\Phi([0, 1)^N \times \M)] \in [0, \infty]\).

For \(E \subset \mathscr{N}(\R^N \times \M)\), define \(E^{(v)} := \{\mu^{(v)} : \mu \in E\}\).  
Given a stationary marked point process \(\Phi\), let \(\mathcal{I}\) be the \(\sigma\)-algebra on \(\mathscr{N}(\R^N \times \M)\) consisting of all sets \(I \subset \mathscr{N}(\R^N \times \M)\) for which
\[
    \mathbb{P}(\Phi \in I \triangle I^{(v)}) = 0
\]
for all \(v \in \R^N\), where \(\triangle\) denotes the symmetric difference.
We say that \(\Phi\) is \emph{ergodic} if \(\mathbb{P}(\Phi \in I) \in \{0, 1\}\) for all \(I \in \mathcal{I}\).

A simple point process on $S$ can be viewed as a special case of a marked point process on $S$ with marks in $\M=\{\ast\}$.
Therefore, the notions defined for marked point processes, such as the ground process, stationarity, intensity and ergodicity, are also defined for simple point processes.

\begin{example}
    An important example of an unmarked point process that is both stationary and ergodic is the homogeneous Poisson point process (see e.g. \cite{last2018lectures}).
    A point process \(\Phi\) on \(\R^N\) is called a \emph{homogeneous Poisson point process with intensity \(\lambda \in (0, \infty)\)} if it satisfies the following two properties:
    \begin{enumerate}[(i)]
        \item For disjoint Borel sets \(B_1, \dots, B_k\), the random variables \(\Phi(B_1), \dots, \Phi(B_k)\) are independent;
        \item For any bounded Borel set \(B\), the random variable \(\Phi(B)\) has a Poisson distribution with mean \(\lambda \cdot \vol{B}\),
    \end{enumerate}
    where \(\vol{A}\) denotes the Lebesgue measure of \(A \subset \R^N\).
    We remark that this notion of intensity is consistent with the definition above:  
    \(\mathbb{E}[\Phi([0, 1)^N)] = \lambda \cdot \vol{[0, 1)^N} = \lambda\).
\end{example}

\subsection{Persistent Homology} \label{subsec: pershom}

We refer the reader to \cite{hatcher2002algebraic} for an overview of the notions of \emph{simplicial complex} and \emph{simplicial homology}.

A \emph{filtration} \(\K = {\{K_t\}_{t\in [0,\infty)}}\) of a simplicial complex \(K\) is an increasing and right-continuous family of complexes indexed by \(\R_{\geq 0}\); that is, \(K_r \subset K_s\subset K\) whenever \(r \leq s\) in \(\R_{\geq 0}\), and for all \(r \in \R_{\geq 0}\), we have \(K_r = \bigcap_{s > r} K_s\).
From the perspective of category theory, a filtration can be viewed as a functor from \(\R_{\geq 0}\) to the category of simplicial complexes and simplicial maps, where \(\R_{\geq 0}\) is viewed as a category whose objects are its elements and in which there exists a unique arrow $r\rightarrow s$ iff $r\leq s$. 

Let us fix a field \(\mathbb{F}\) and a homology degree \(q\). Let \(\mathbf{Vec}_{\mathbb{F}}\) be the category of vector spaces and linear maps over \(\mathbb{F}\).
By applying the homology functor \(H_q(-; \mathbb{F})\) to the filtration $\K$, we obtain the functor
\(    H_q(\K; \mathbb{F}): \R_{\geq 0} \to \mathbf{Vec}_{\mathbb{F}}.
\)
In general, a functor from \(\R_{\geq 0}\) to \(\mathbf{Vec}_{\mathbb{F}}\) is called a \emph{persistence module}. The direct sum of two persistence modules is defined pointwise.
The basic building blocks of a persistence module are the \emph{interval modules}: For any interval $J\subset \R_{\geq 0}$, the interval module $\mathbb{I}_J$ is defined as follows:
\[
    \mathbb{I}_J(r) =
    \begin{cases}
        \mathbb{F}, & \text{if } r \in J, \\
        0, & \text{otherwise},
    \end{cases}
    \quad \text{and} \quad
    \mathbb{I}_J(r \leq s) =
    \begin{cases}
        \mathrm{Id}_{\mathbb{F}}, & \text{if } r\leq s \in J, \\
        0, & \text{otherwise}.
    \end{cases}
\]
Two persistence modules \(P\) and \(Q\) are called \emph{isomorphic}, written \(P \cong Q\), if there exists a natural isomorphism between them.

Given a persistence module \(P\), if $\dim P(r)<\infty$ for all $r\in \R_{\geq 0}$, \(P\) is decomposed into a direct sum of interval modules, which is unique up to isomorphism \cite[Theorem 1.1]{crawley-boevey2015}.
This implies that, for any filtration $\K$ of a finite simplicial complex, there exists a unique multiset of points \((b_i, d_i)\)
 in  
\begin{equation}\label{eq:Delta}
\Delta := \{(x, y) \in \overline{\R_{\geq 0}}^2 : 0 \leq x < y \leq \infty\} \mbox{\ (see Figure \ref{subfig: delta})}
\end{equation} 
such that 
\[
    H_q(\K,\mathbb{F}) \cong \bigoplus_i \mathbb{I}_{[b_i, d_i)},
\]
 where \(\overline{\R_{\geq 0}} = \R_{\geq 0}\cup\{\infty\}= [0, \infty]\).
The above multiset 
is called the degree-\(q\) \emph{persistence diagram} of \(\K\), and is denoted by \(D_q(\K)\), which summarizes the evolution of $q$-dimensional homological features 
(i.e., \(q\)-dimensional holes)
in $\K$ over the indexing set $\R_{\geq 0}$. 
Specifically, if \((b, d)\in D_q(\K)\), then, in $\K$, there exists a $q$-cycle formed at time $b$ that becomes a $q$-boundary at time $d$.
(see Figure~\ref{fig: filtration-diagram}).

\begin{figure}
    \centering
    \begin{subfigure}{0.3\linewidth}
        \includegraphics[width=\linewidth]{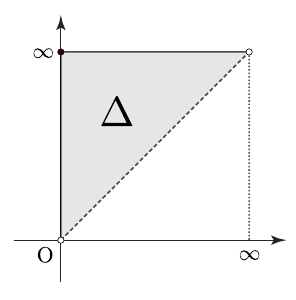}
        \caption{}
        \label{subfig: delta}
    \end{subfigure}
    \begin{subfigure}{0.3\linewidth}
        \includegraphics[width=\linewidth]{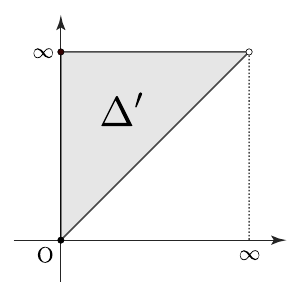}
        \caption{}
        \label{subfig: delta'}
    \end{subfigure}
    \caption{
        Illustration of $\Delta$ and $\Delta'$. The set \(\Delta'\) is defined by adding the half-open diagonal  \(\{(t, t): 0 \leq t < \infty\}\) to \(\Delta\).
        Note that the point \((\infty, \infty)\) does not belong to \(\Delta'\), which makes  \(\Delta'\) noncompact.
    }
    \label{fig: delta-and-delta'}
\end{figure}

We now introduce another key concept of persistent homology. 
For any pair \(r \leq s\) in \(\R_{\geq 0}\), the degree-\(q\) \emph{persistent Betti number} \(\beta_q^{r, s}(\K)\) is defined as the rank of the linear map \(H_q(\iota^{r, s}; \mathbb{F})\), where \(\iota^{r, s}: K_r \hookrightarrow K_s\) is the inclusion, and can also be described as follows:
Given a simplicial complex $K$, let \(\mathcal{Z}_q(K)\) and \(\mathcal{B}_q(K)\) be the subgroups of \(q\)-cycles and \(q\)-boundaries respectively of the simplicial \(q\)-chain group of $K$ with coefficients in $\mathbb{F}$. The image of the map \(H_q(\iota^{r, s}; \mathbb{F})\) is isomorphic to the space $\frac{\mathcal{Z}_q(K_r)}{\mathcal{Z}_q(K_r) \cap \mathcal{B}_q(K_s)}$ of homology classes of $K_r$ that are still alive at $K_s$. Therefore, we have 
\begin{equation}
    \beta_q^{r, s}(\K) = \dim \frac{\mathcal{Z}_q(K_r)}{\mathcal{Z}_q(K_r) \cap \mathcal{B}_q(K_s)}. \label{eq: persbetti}
\end{equation}

\subsection{Verbose and Concise Diagrams} \label{subsec: verbose}

The persistence diagram is enriched into the \emph{verbose diagram} \cite{usherzhang2016}, which we briefly review in this section.
The verbose diagram has a rather abstract and technical definition, although its interpretation is as simple as shown in Figure~\ref{fig: filtration-diagram}. 
We remark that this technicality is 
exploited when establishing our main theorems.
We mostly follow the terminology of \cite{mémolizhou2024}, which is relatively simpler than that of \cite{usherzhang2016}.

\paragraph{Filtered chain complexes and their verbose and concise diagrams.} Let us fix a field \(\mathbb{F}\).
A \emph{non-Archimedean normed vector space} is a pair \((C, \ell)\), where \(C\) is an \(\mathbb{F}\)-vector space and \(\ell: C \to [-\infty, \infty)\) is a \emph{vector filtration function}, meaning that it satisfies the following properties:
\begin{itemize}
    \item \(\ell(x) = -\infty\) if and only if \(x = 0\);
    \item For any \(0 \neq \lambda \in \mathbb{F}\) and \(x \in C\), we have \(\ell(\lambda x) = \ell(x)\);
    \item For any \(x, y \in C\), we have \(\ell(x + y) \leq \max\{\ell(x), \ell(y)\}\).
\end{itemize}
A finite set \(\{x_1, \dots, x_r\} \subset C\) is said to be \emph{orthogonal} if, for any \(\lambda_1, \dots, \lambda_r \in \mathbb{F}\),
\[
    \ell\left(\sum_{i = 1}^r \lambda_i x_i\right) = \max \{\ell(x_i) : \lambda_i \neq 0\},
\]
where we adopt the convention that the maximum of the empty set is \(-\infty\).
A non-Archimedean normed vector space is called \emph{orthogonalizable} if it admits an orthogonal basis.

Suppose that \((C, \ell_C)\) and \((D, \ell_D)\) are two finite-dimensional orthogonalizable vector spaces.  
Let \(F: C \to D\) be a linear map with rank \(r\). A \emph{singular value decomposition} of \(F\) is a pair\\ \(((y_1, \dots, y_n), (x_1, \dots, x_m))\), where \((y_1, \dots, y_n)\) is an ordered orthogonal basis for \(C\), and \((x_1, \dots, x_m)\) is an ordered orthogonal basis for \(D\), satisfying the following conditions:
\begin{itemize}
    \item \((y_{r+1}, \dots, y_n)\) is an ordered orthogonal basis for \(\mathrm{Ker}\, F\);
    \item \((x_1, \dots, x_r)\) is an ordered orthogonal basis for \(\mathrm{Im}\, F\);
    \item \(F(y_i) = x_i\) for all \(i = 1, \dots, r\).
\end{itemize}
Every linear map between two finite-dimensional orthogonalizable spaces admits a singular value decomposition \cite[Theorem~3.4]{usherzhang2016}.

A \emph{filtered chain complex (FCC)} is a triple \(\mathcal{C} = (C_*, \partial, \ell)\) that satisfies the following conditions:
\begin{itemize}
    \item \((C_*, \partial)\) is a chain complex such that \(\ell \circ \partial \leq \ell\);
    \item Let \(C_* = \bigoplus_{q \in \mathbb{Z}} C_q\); then \(\sum_{q \in \mathbb{Z}} \dim C_q < \infty\), and each \((C_q, \ell|_{C_q})\) is an orthogonalizable vector space.
\end{itemize}
For an FCC \(\mathcal{C} = (C_*, \partial, \ell)\) and each \(q \in \mathbb{Z}\), define \(\partial_q: C_q \to \mathrm{Ker}\, \partial_{q-1}\) as the restriction of \(\partial\).  
Consider a singular value decomposition \(((y_1, \dots, y_n), (x_1, \dots, x_m))\) of the linear map \(\partial_{q+1}\).
If the rank of \(\partial_{q+1}\) is \(r\), then the degree-\(q\) \emph{verbose diagram} of \(\mathcal{C}\) is the multiset \(D_{\mathrm{Ver}, q}(\mathcal{C})\) consisting of:
\begin{itemize}
    \item a pair \((\ell(x_i), \ell(y_i))\) for each \(i = 1, \dots, r\); and
    \item a pair \((\ell(x_i), \infty)\) for each \(i = r+1, \dots, m\).
\end{itemize}
The \emph{lifetime} of a point \((b, d) \in D_{\mathrm{Ver}, q}(\mathcal{C})\) is defined as the value \(d - b \in [0, \infty]\).  
The \emph{concise} diagram of \(\mathcal{C}\) is defined as the submultiset of the verbose diagram consisting of the points with strictly positive lifetime.

For the above notions to be well defined, the verbose diagram must be independent of the choice of singular value decomposition; this is indeed guaranteed by
\cite[Theorem~7.1]{usherzhang2016}.
It is noteworthy that the verbose diagram determines the filtered chain isomorphism type of an FCC whereas the concise diagram determines only the filtered chain homotopy type of it  \cite[Theorems A and B]{usherzhang2016}.

The concise diagram can be identified with the standard persistence diagram defined in Section~\ref{subsec: pershom}, provided the following mild assumptions hold.  
Let \(\K = {\{K_t\}_{t\in [0,\infty)}}\)
be a filtration of a simplicial complex \(K\).  
We assume that:  
(i) \(K\) is finite, and  
(ii) every simplex of \(K\) appears in the filtration, i.e.
\begin{equation} \label{eq: filtrationassumption(ii)}
    K = \bigcup_{t \geq 0} K_t.
\end{equation}
Note that any filtration can be modified to satisfy assumption~(ii) by redefining \(K\) as the right-hand side of \eqref{eq: filtrationassumption(ii)}.
Assumption~(i) ensures that all homology groups are finite-dimensional, so that the persistence diagrams are well-defined. 
Assumption (ii) is necessary for inducing an FCC from the filtration $\K$, as explained below.
 
Let \((C_*(K), \partial)\) be the simplicial chain complex of \(K\) over the coefficient field \(\mathbb{F}\).  
Define a function \(\ellK{\K}: C_*(K) \to [-\infty, \infty)\) by
\begin{equation} \label{eq: vectfilt}
    \ellK{\K} \left( \sum_i \lambda_i \sigma_i \right) = \max_i \{\tK{\K}(\sigma_i): \lambda_i \neq 0\},
\end{equation}
where the \(\sigma_i\) are simplices of \(K\), and the \emph{birth-time function}
\(\tK{\K}\) is defined by, for any simplex \(\sigma \subset K\),
\begin{equation} \label{eq: birthtime}
   \tK{\K}(\sigma) = \inf \{t \geq 0: \sigma \in K_t\}.
\end{equation}

\begin{remark}\label{rmk:meaning of ell}
For any chain \(x = \sum_i \lambda_i \sigma_i\), \(\ellK{\K}(x)\) is the least parameter value at which each \(\sigma_i\) with \(\lambda_i \neq 0\) exists in the filtration.
In other words, \(\ellK{\K}(x)\) is the birth-time of the chain \(x\).
\end{remark}
Note that \(\tK{\K}(\sigma)<\infty\) for every $\sigma\in K$ due to assumption~(ii), which ensures that the image of $\ellK{\K}$ is contained in $[-\infty,\infty)$.
Since the set of \(q\)-simplices form an orthogonal basis for the \(q\)-chain group for every degree \(q \geq 0\), we have that
\begin{equation} \label{eq: filtrationFCC}
    \mathcal{C}(\K) := (C_*(K), \partial, \ellK{\K})
\end{equation}
forms an FCC.

As a result, the verbose diagram \(D_{\mathrm{Ver}, q}(\mathcal{C}(\K))\) is well-defined.
Furthermore, the degree-\(q\) concise diagram of \(\mathcal{C}(\K)\) coincides with the degree-\(q\) persistence diagram \(D_q(\K)\), and thus \(D_q(\K)\) can be viewed as a subset of \(D_{\mathrm{Ver}, q}(\mathcal{C}(\K))\) \cite[Theorem~6.2 and Definition 6.3]{usherzhang2016}.
The points of \(D_{\mathrm{Ver}, q}(\mathcal{C}(\K))\) that do not belong to \(D_q(\K)\) lie on the diagonal line $y=x$ in the plane (e.g. the point $(3,3)$ in Figure \ref{fig: filtration-diagram}).
Given a singular value decomposition, each of such points correspond to two chains \({y} \in C_{q+1}\) and \({x} \in \mathrm{Ker}\, \partial_q\) such that \(\partial_{q+1}({y}) = {x}\)
and \(\ellK{\K}(x) = \ellK{\K}(y)\).
By Remark~\ref{rmk:meaning of ell}, each of the points on the diagonal in the verbose diagram represents a \(q\)-cycle (\(x\)) that becomes boundary (\(\partial_{q+1}(y) = x\)) as soon as it is born (\(\ellK{\K}(x) = \ellK{\K}(y)\)).

\paragraph{Metrics for FCCs and verbose diagrams.} As a counterpart of the bottleneck distance for persistence diagrams, the \emph{matching distance} between verbose diagrams was introduced in \cite{mémolizhou2024}, and we review it below.

We first introduce a metric \(d_\infty\) on \(\Delta'\), induced by the \(l_\infty\) norm.
For \((p, q), (p', q') \in \Delta'\), define
\begin{align*}
    d_\infty((p, q), (p', q')) &:= \|(p, q) - (p', q')\|_\infty \\
    &=
    \begin{cases}
        \max \{|p - p'|, |q - q'|\}, &\quad \text{if } q, q' < \infty, \\
        |p - p'|, &\quad \text{if } q = q' = \infty, \\
        \infty, &\quad \text{otherwise}.
    \end{cases}
\end{align*}

Next, we review the notions of multisets and maps between them.
A \emph{multiset} \(A\) is an ordered pair \((U, m)\), where \(U\) is a set called the \emph{underlying set}, and \(m: U \to \mathbb{Z}_{\geq 0}\) is a function called the \emph{multiplicity function}.  
The \emph{support} of \(A\) is defined as \(\mathrm{supp}(A) := \{a \in U : m(a) > 0\}\), and the \emph{cardinality} \(|A|\) of \(A\) is given by the sum of multiplicities over the support.
A multiset \(A\) can also be represented as the \emph{labeled set}
\[
    \tilde{A} := \{(a, i) : a \in \mathrm{supp}(A),\, 1 \leq i \leq m(a)\}.
\]
A \emph{map} \(\phi\) between two multisets \(A\) and \(B\) is defined as a function \(\phi: \tilde{A} \to \tilde{B}\) between their labeled sets.  
We denote this map by \(\phi: A \to B\) with slight abuse of notation.  
Such a map is called a \emph{bijection} if it is a bijection between the labeled sets.

\begin{definition} \label{def: matchingdist}
    Let \(A\) and \(B\) be multisets supported on \(\Delta'\).
    The \emph{matching distance} between \(A\) and \(B\) is defined by
    \[
        d_M(A, B) = \inf \left\{ \sup_{a \in A} \|a - \phi(a)\|_\infty : A \xrightarrow{\phi} B \text{ is a bijection} \right\},
    \]
    where \(d_M(A, B) = \infty\) if \(|A| \neq |B|\).
\end{definition}

Different notions of \emph{interleaving distance} for FCCs have been introduced in \cite{mémolizhou2024, usherzhang2016}; here, we review the one from \cite{mémolizhou2024}.

For an FCC \((C_*, \partial_C, \ell_C)\) and \(\lambda \in \R\), let
\[
    C_*^\lambda := \ell_C^{-1}([-\infty, \lambda]) \subset C_*
\]
be the subspace of \(C_*\) consisting of those \(x \in C_*\) for which \(\ell_C(x) \leq \lambda\).

\begin{definition}[{\cite[Definition~4.1~and~Remark~4.2]{mémolizhou2024}}] \label{def: interleavingdist}
    For \(\delta \geq 0\), a \emph{\(\delta\)-interleaving} between two FCCs \((C_*, \partial_C, \ell_C)\) and \((D_*, \partial_D, \ell_D)\) is a pair \((\mathcal{F}_*, \mathcal{G}_*)\) of chain maps \(\mathcal{F}_*: C_* \to D_*\) and \(\mathcal{G}_*: D_* \to C_*\) such that:
    \begin{itemize}
        \item \(\ell_D \circ \mathcal{F}_* \leq \ell_C + \delta\);
        \item \(\ell_C \circ \mathcal{G}_* \leq \ell_D + \delta\);
        \item For each \(\lambda \in \R\), the compositions
        \[
            C_*^\lambda \xrightarrow{\mathcal{F}_*^\lambda} D_*^{\lambda + \delta} \xrightarrow{\mathcal{G}_*^{\lambda + \delta}} C_*^{\lambda + 2\delta}
            \quad \text{and} \quad
            D_*^\lambda \xrightarrow{\mathcal{G}_*^\lambda} C_*^{\lambda + \delta} \xrightarrow{\mathcal{F}_*^{\lambda + \delta}} D_*^{\lambda + 2\delta}
        \]
        are equal to the respective inclusion maps, where \(\mathcal{F}_*^a = \mathcal{F}_*|_{C_*^a}\) and \(\mathcal{G}_*^b = \mathcal{G}_*|_{D_*^b}\) for each \(a, b \in \R\).
    \end{itemize}
    We say that two FCCs are \emph{\(\delta\)-interleaved} if there exists a \(\delta\)-interleaving between them.
    For two FCCs \(\mathcal{C}\) and \(\mathcal{D}\), the \emph{interleaving distance} between \(\mathcal{C}\) and \(\mathcal{D}\) is defined by
    \[
        d_I(\mathcal{C}, \mathcal{D}) := \inf \{ \delta \geq 0 : \mathcal{C} \text{ and } \mathcal{D} \text{ are } \delta\text{-interleaved} \}.
    \]
    Here we follow the convention that \(\inf \varnothing = \infty\).
\end{definition}

The following isometry theorem establishes a tight connection between the matching distance of verbose diagrams and the interleaving distance of FCCs:

\begin{theorem}[{\cite[Theorem~2]{mémolizhou2024}}] \label{thm: isometry}
    For any FCCs \(\mathcal{C}\) and \(\mathcal{D}\),
    \[
        \sup_{q \geq 0} d_M(D_{\mathrm{Ver}, q}(\mathcal{C}), D_{\mathrm{Ver}, q}(\mathcal{D})) = d_I(\mathcal{C}, \mathcal{D}).
    \]
\end{theorem}

\section{Previous Results} \label{sec: prevresults}

This section reviews the main results of \cite{hiraokaetal2018} and \cite{shirai2022}.
We begin by (1) interpreting persistence diagrams and verbose diagrams as measures,
(2) introducing a notion of filtration function, 
and (3) outlining a method for generating random point clouds.  

\paragraph{Persistence diagrams as measures.}

Recall the set $\Delta$ in Equation \eqref{eq:Delta}. By adjoining the half-open diagonal with $\Delta$, we obtain \[
    \Delta' := \{(x, y) \in \overline{\R_{\geq 0}}^2: 0 \leq x \leq y \leq \infty\} \setminus \{(\infty, \infty)\},
\]
as illustrated in Figure \ref{subfig: delta'}. We equip \(\overline{\R_{\geq 0}} = [0, \infty]\) with the order topology, and topologize each of \(\Delta\) and \(\Delta'\) with the subspace topology inherited from \(\overline{\R_{\geq 0}}^2\).

\begin{convention}\label{conv:PD as a measure}
Given a filtration $\K$ of a finite simplicial complex, we identify the persistence diagram \(D_q(\K)\) with the discrete measure
\(
    \xi_q(\K) = \sum_{p \in D_q(\K)} \delta_p
\)
on $\Delta$ so that \(\xi_q(\K) \in \mathscr{N}(\Delta)\).

Similarly, given any FCC \(\mathcal{C}\), we regard the verbose diagram \(D_{\mathrm{Ver}, q}(\mathcal{C})\) as the discrete measure
\(
    \xi_{\mathrm{Ver}, q}(\mathcal{C}) = \sum_{p \in D_{\mathrm{Ver}, q}(\mathcal{C})} \delta_p
\)
on $\Delta'$, so that \(\xi_{\mathrm{Ver}, q}(\mathcal{C}) \in \mathscr{N}(\Delta')\).
\end{convention}

We remark that for $r\leq s\in \R_{\geq0}$, \begin{equation} \label{eq: persbettimeasure}
    \beta_q^{r, s}(\K) = \xi_q(\K)([0, r] \times (s, \infty]), 
\end{equation}
as depicted in Figure~\ref{fig: persistent-betti-number-measure}. This equality is called the \emph{Fundamental Lemma of Persistent Homology} \cite[Chapter VII]{edelsbrunner2010}. We will extend this lemma to the setting of verbose diagrams in a later section.

\paragraph{\v{C}ech and Vietoris--Rips filtrations and their Mark-specific variations.} 

A filtration \(\K = {\{K_t\}_{t\in [0,\infty)}}\) of a simplicial complex \(K\) (satisfying \eqref{eq: filtrationassumption(ii)}) can be described by means of the {birth-time} function \(t_{\K}: K \to [0, \infty)\) defined in Equation~\eqref{eq: birthtime}.
For example, in the case of the \v{C}ech or Vietoris--Rips filtrations, 
the vertex set of \(K\) is a finite subset \(V\) of \(\R^N\), and their {birth-time} functions are
\begin{align*}
    t_C(\sigma) &= \inf_{w \in \R^N} \max_{x \in \sigma} \|x - w\|,\ \forall \sigma(\neq \emptyset)\subset V,\ \mbox{and} \\
    t_R(\sigma) &= \mathrm{diam}(\sigma) = \max_{x_1, x_2 \in \sigma} \|x_1 - x_2\|,\ \forall \sigma(\neq \emptyset)\subset V
\end{align*}
respectively, where \(\|\cdot\|\) denotes the Euclidean norm.

These filtrations can be modified when each point is equipped with 
a mark (cf. Section~\ref{subsec: measure}).
For example, given an atomic configuration, suppose that we regard the center of each atom as a vertex in \(\R^3\) and build a similar filtration on these points \cite{hiraoka2016hierarchical}.
It is natural to take the radius of each atom as its mark.
More concretely, let \(R_0\) be the maximum among the radii of all atoms.  
Then the vertex set is a finite subset \(V\) of \(\R^3 \times [0, R_0]\).
Let \(K'\) be the full simplicial complex on \(V\), i.e., the set of all nonempty subsets of \(V\).
In this case, the functions \(t_C', t_R': K' \to [0, \infty)\) defined by
\begin{equation} \label{eq: cechvrvariant}
    \begin{split}
        t_C'(\sigma) &= \inf_{w \in \R^3} \max_{(x, r) \in \sigma} (\|x - w\| - r)^+, \\
        t_R'(\sigma) &= \max_{(x_1, r_1), (x_2, r_2) \in \sigma} (\|x_1 - x_2\| - r_1 - r_2)^+,
    \end{split}
\end{equation}
where \(a^+ = \max\{a, 0\}\) for \(a \in \R\), yield modified versions of the \v{C}ech and Vietoris--Rips filtrations of the given atomic configuration, respectively.
Let \(\overline{B}(x, r)\) denote the closed ball in \(\R^3\) of radius \(r\) centered at \(x\).  
Then \(t_C'(\sigma)\) computes the radius of the smallest closed ball that intersects every \(\overline{B}(x, r)\) with \((x, r) \in \sigma\),  
while \(t_R'(\sigma)\) computes the largest pairwise distance between these closed balls.
We remark that there are other variations of \v{C}ech and Vietoris--Rips filtrations with marks: see e.g. \cite[Examples~ 2.1-2.3]{shirai2022}.

Notice that the values of 
\(t_C'(\sigma)\) and \(t_R'(\sigma)\) depend 
only on the vertices of the input simplex \(\sigma\), not on the other points in the underlying vertex set.
Also, the right-hand sides of Equation~\eqref{eq: cechvrvariant} are defined for any simplex \(\sigma\) whose vertices lie in \(\R^3 \times [0, R_0]\).  
Thus, the domains of \(t_C'\) and \(t_R'\) can be extended to the collection of all nonempty finite subsets $\sigma$ of \(\R^3 \times [0, R_0]\).
Moreover, the ambient space does not need to be three-dimensional, and \(R_0 \geq 0\) can be arbitrary.

For any set \(S\), let \(\mathscr{F}(S)\) denote the set of all nonempty finite subsets of \(S\). 
The functions \(\kappa_C', \kappa_R': \mathscr{F}(\R^N \times [0, R_0]) \to [0, \infty)\) defined by
\begin{equation}\label{eq:kC and kR}
\begin{aligned}
    {\kappa_C'}(\sigma) &= \inf_{w \in \R^N} \max_{(x, r) \in \sigma} (\|x - w\| - r)^+, \\
    {\kappa_R'}(\sigma) &= \max_{(x_1, r_1), (x_2, r_2) \in \sigma} (\|x_1 - x_2\| - r_1 - r_2)^+,
\end{aligned}
\end{equation}
yield modified \v{C}ech and Vietoris--Rips filtrations 
on any finite set $V\subset\R^N \times [0, R_0]$. 
If \(R_0 = 0\), then $\kappa_C'$ and $\kappa_R'$ respectively reduce to the functions that induce the standard \v{C}ech and Vietoris--Rips filtrations on \(V \subset \R^N \times \{0\} \simeq \R^N\).
\begin{definition}[Filtration functions for Standard \v{C}ech and Vietoris--Rips filtrations]
\label{def:VR and Cech}
We denote these special cases for \(\kappa_C'\) and \(\kappa_R'\) respectively by $\kappa_C$ and $\kappa_R$, which are functions $\mathscr{F}(\R^N) \to [0, \infty)$.
\end{definition}

\paragraph{Simplicial filtrations induced from marked point sets in $\R^N$.}
 
A function \(f\) on \(\mathscr{F}(\R^N)\) (resp. \(\mathscr{F}(\R^N \times \M)\)) is called \emph{measurable} if, for each \(k \geq 1\), there exists a measurable function \(\tilde{f}_k\) on \((\R^N)^k\) (resp. \((\R^N \times \M)^k\)) such that
\[
    f(\{x_1, \dots, x_k\}) = \tilde{f}_k(x_1, \dots, x_k),
\]
where \(\tilde{f}_k\) is invariant under permutations of its arguments.

\begin{definition}\label{def:filtration function}
    A measurable function \(\kappa: \mathscr{F}(\R^N ) \to [0, \infty]\) is called a \emph{filtration function (for unmarked point sets)} 
    if it satisfies the following:
    \begin{itemize}
        \item[(K1)] (Monotonicity) \(\kappa(\sigma) \leq \kappa(\tau)\) whenever \(\sigma \subset \tau\);\label{item:filtration function-monotone}
        \item[(K2)] (Translation-invariance) \(\kappa(\sigma) = \kappa(\sigma + x)\) for all \(x \in \R^N\), where \(\sigma + x := \{y + x: y \in \sigma\}\);\label{item:filtration function-translation}
        \item[(K3)] (Regularity) There exists an increasing function \(\rho: [0, \infty] \to [0, \infty]\) with \(\rho(t) < \infty\) for all finite \(t\), such that
        \[
            \|x - y\| \leq \rho(\kappa(\{x,y\}))
        \]
        for all \(x, y \in \R^N \).\label{item:filtration function-regularity}
    \end{itemize}
    
    Similarly, a measurable function \(\kappa: \mathscr{F}(\R^N \times \M) \to [0, \infty]\) is also called a \emph{filtration function (for marked point sets)} 
    if it satisfies (K1) and
    \begin{itemize}
        \item[(K2')] \(\kappa\) is invariant under translations acting on $\R^N$, i.e., \(\kappa(\sigma) = \kappa(\sigma + x)\) for all \(x \in \R^N\), where \(\sigma + x := \{(y + x, m): (y, m) \in \sigma\}\);
        \item[(K3')] There exists an increasing function \(\rho: [0, \infty] \to [0, \infty]\) with \(\rho(t) < \infty\) for all finite \(t\), such that
        \[
            \|x - y\| \leq \rho(\kappa(\{(x, m), (y, n)\}))
        \]
        for all \((x, m), (y, n) \in \R^N \times \M\).
    \end{itemize}
    
\end{definition}
We remark that condition (K3) (resp. (K3')) guarantees that any simplex containing two vertices $x$ and $y$ (resp. $(x,m)$ and $(y,n)$) that are far from each other appears sufficiently late in the filtration. 

\begin{example} \label{ex: modifiedcechvr}
    \(\kappa_C'\) and \(\kappa_R'\) in \eqref{eq:kC and kR} are filtration functions: First, they clearly satisfy conditions (K1) and (K2').
    Condition (K3') also holds by taking \(\rho(t) = 2t + 2R_0\) for \(\kappa_C'\), and \(\rho(t) = t + 2R_0\) for \(\kappa_R'\).
\end{example}
Let \(\pi: \R^N \times \M \to \R^N\) denote the projection onto the first component. We call \(X \in \mathscr{F}(\R^N \times \M)\)  a \emph{simple marked point set} if for each \(x \in \R^N\), the cardinality of \(X \cap \pi^{-1}\{x\}\) is at most one.

Any filtration function $\kappa$ for marked point sets induces a simplicial filtration on any simple marked point set $X$  defined as
\(
    \K^\kappa(X) := \{K^\kappa(X, t)\}_{t \geq 0},
\)
where
\begin{equation} \label{eq: sublevel}
    K^\kappa(X, t) := \{\sigma \subset X : \kappa(\sigma) \leq t\}.
\end{equation}
Similarly, any filtration function $\kappa$ for unmarked point sets induces a simplicial filtration on any $X\in \mathscr{F}(\R^N)$, denoted also by $\K^\kappa(X)$.
We remark that, for any simplex \(\sigma\) that appears in \(\K^{\kappa}(X)\), 
    \begin{equation} \label{eq: filtfuncandbirthtime}
        \kappa(\sigma) = \tK{\K^{\kappa}(X)}(\sigma) \quad  \mbox{(cf. Equation~\eqref{eq: birthtime}).}
    \end{equation}

\begin{definition} \label{def: kappafilt}
   In either the marked or the unmarked setting, we refer to \(\K^\kappa(X)\), defined above, as the \emph{\(\kappa\)-filtration} on \(X\).
\end{definition}

We remark that simplices with infinite \(\kappa\)-value 
do not belong to 
the final complex \(\bigcup_{t\geq 0}K^\kappa(X,t)\) and thus the final complex may not be the full simplicial complex on \(X\). 

The following definition will be useful later on.

\begin{definition} \label{def: realizable}
      Let \(\kappa\) be a filtration function for unmarked point sets (resp. marked point sets) and \(q \geq 0\) an integer.
    A point \((b, d) \in \Delta\) is said to be \emph{\((\kappa, q)\)-realizable by} \(\Xi \in \mathscr{F}(\R^N)\) (resp. a simple marked point set \(\Xi\)
    ) if 
    \(
        (b, d) \in D_q(\K^\kappa(\Xi)),
    \)
    or equivalently, if
    \[
        \xi_q(\K^\kappa(\Xi))(\{(b, d)\}) \geq 1\  \mbox{(cf. Convention~\ref{conv:PD as a measure})}.
    \]
    A point that is \((\kappa, q)\)-realizable by some \(\Xi\) is simply called \emph{\((\kappa, q)\)-realizable}.
    We omit \((\kappa, q)\) and simply write “realizable” when \(\kappa\) and \(q\) are clear from the context.
    We denote by \(R_q = R_q(\kappa) \subset \Delta\) the set of all \((\kappa, q)\)-realizable points.
    
\end{definition}

For example, it is known from \cite[Example~1.8]{hiraokaetal2018} that \(R_q(\kappa_C)\) and \(R_q(\kappa_R)\) (see Definition~\ref{def:VR and Cech}) are cones in \(\Delta\) with apex at the origin, and 
\begin{equation} \label{eq: cechrealizable}
    R_q(\kappa_C) =
    \begin{cases}
        \{0\} \times (0, \infty], &\quad \text{if } q = 0, \\
        \{(b, d) \in \Delta : 0 < b < d < \infty\}, &\quad \text{if } q = 1, 2, \dots, N-1, \\
        \varnothing, &\quad \text{if } q \geq N.
    \end{cases}
\end{equation}

\paragraph{Generation of random point clouds.}

To specify a growing observation window for a point process in $\R^N$, we define a \emph{convex averaging sequence} in \(\mathbb{R}^N\) as a sequence \(\mathcal{L} = \{\Lambda_n\}_n\) of bounded Borel sets in \(\mathbb{R}^N\) satisfying the following conditions:
\begin{itemize}
    \item \(\Lambda_n\) is convex for each $n$;
    \item \(\Lambda_n \subset \Lambda_m\) whenever \(n \leq m\);
    \item \(\displaystyle \sup_n r(\Lambda_n) = \infty\), where \(r(\Lambda)\) denotes the supremum of the radii of open balls contained in \(\Lambda\).
\end{itemize}

The \emph{restriction} of a marked point process \(\Phi\) on \(\R^N\) with marks in \(\M\) to a bounded Borel set \(\Lambda \subset \R^N\) is defined as  
\[
    \Phi_\Lambda(\cdot) := \Phi(\cdot \cap (\Lambda \times \M)).
\]
By the definition of a marked point process, the restricted process \(\Phi_\Lambda\) almost surely produces 
a finite sum $\sum\delta_{(x_i,m_i)}$ of Dirac measures where $\{(x_i,m_i)\in \Lambda\times \M\}$ is a simple marked point set. 
Therefore, via the identification \(\sum\delta_{(x_i,m_i)} \leftrightarrow \{(x_i,m_i)\in \Lambda\times \M\}, \) the process \(\Phi_\Lambda\) is almost surely 
identified with 
a simple marked point set. 
More specifically,  \(\Phi_\Lambda\) is identified with a simple marked point set on the event that
the ground process \((\Phi_\Lambda)_g\) produces no repeated points.

\begin{notation}Given a filtration function \(\kappa\) for marked point sets 
and an integer \(q \geq 0\), let \(\xi_{q,\Lambda}^\kappa\) and \(\xi_{\mathrm{Ver}, q, \Lambda}^\kappa\) respectively denote the degree-\(q\) persistence  and verbose diagrams of the \(\kappa\)-filtration on \(\Phi_\Lambda\), i.e., 
\[
    \xi_{q,\Lambda}^\kappa := \xi_q(\K^\kappa(\Phi_\Lambda))\mbox{ and } \xi_{\mathrm{Ver}, q, \Lambda}^\kappa := \xi_{\mathrm{Ver}, q}(\mathcal{C}(\K^\kappa(\Phi_\Lambda))). 
\]
Also, for any $r\leq s\in \R$, let $\beta_{q, \Lambda}^{\kappa, r, s}$ denote the degree-$q$ persistent Betti number of the \(\kappa\)-filtration on \(\Phi_\Lambda\), i.e.
\[
    \beta_{q, \Lambda}^{\kappa, r, s} := \beta_q^{r, s}(\K^\kappa(\Phi_\Lambda)).
\]
We use \(\xi_{q,\Lambda}^\kappa\), $\xi_{\mathrm{Ver}, q, \Lambda}^\kappa$,  and \(\beta_{q, \Lambda}^{\kappa, r, s}\) even when  \(\kappa\)  is a filtration function for unmarked point sets and \(\Phi\) is an unmarked point process on \(\R^N\).
\end{notation}
\begin{remark}\label{rmk:point processes}
Conditioned on  \((\Phi_\Lambda)_g\) producing no repeated points, \(\xi_{q,\Lambda}^\kappa\) and \(\xi_{\mathrm{Ver}, q, \Lambda}^\kappa\) are point processes on \(\Delta\) and $\Delta'$ respectively (see Figure~\ref{fig: delta-and-delta'}), and \(\beta_{q, \Lambda}^{\kappa, r, s}\) is an integer-valued random variable. The proof of these claims is rather lengthy, and thus we omit it.
\end{remark}

\paragraph{Main results of \cite{hiraokaetal2018} and \cite{shirai2022}.}

We first present the main results of \cite{hiraokaetal2018}, which concern only unmarked point processes on \(\R^N\).

The first theorem asserts that, for $\Lambda_n = [-n/2, n/2)^N$, the asymptotic behavior of \(\xi_{q,\Lambda_n}^\kappa\) is  deterministic.  
More precisely, the normalized version of \(\xi_{q,\Lambda_n}^\kappa\) converges vaguely to a Radon measure on \(\Delta\) as \(n \to \infty\).

\begin{theorem}[{\cite[Theorem~1.5]{hiraokaetal2018}}] \label{thm: vagueconv}
    Assume that \(\Phi\) is a stationary point process on \(\R^N\)
    and let \(\kappa\) be a filtration function.
    For \(n \geq 1\), let \(\Lambda_n = [-n/2, n/2)^N\).
    Then, for each \(q \geq 0\), there exists a unique Radon measure \(\nu_q = \nu_q(\kappa,\Phi)\) on \(\Delta\) such that,
    \[
        \frac{1}{n^N}\,\mathbb{E}[\xi_{q,\Lambda_n}^\kappa] \xrightarrow{v} \nu_q
        \quad\text{as } n \to \infty.
    \]
    Moreover, if \(\Phi\) is ergodic, then, almost surely,
    \[
        \frac{1}{n^N}\,\xi_{q,\Lambda_n}^\kappa \xrightarrow{v} \nu_q
        \quad\text{as } n \to \infty.
    \]
\end{theorem}

For $x\in \R^N$ and $r\in \R_{\geq0}$, let $B(x,r)$ denote the open ball of radius $r$ centered at $x$. For compact subsets \(A\) and \(B\) of \(\R^N\), the \emph{Hausdorff distance} between \(A\) and \(B\) \cite{burago2001} is defined as
\[
    d_H(A, B) = \inf\{r\geq 0: A^r\supset B \mbox{ and } B^r\supset A \}
    ,
\]
where $A^r=\bigcup_{a\in A}B(a,r)$.
Given any point process \(\Phi\) on a second-countable, locally compact Hausdorff space \(S\), by \(\Theta_{\Phi}\), let us denote the probability distribution of \(\Phi\) on \(\mathscr{N}(S)\).
Also, given any Borel set \(\Lambda \subset \R^N\),
let \(r_\Lambda: \mathscr{N}(\R^N) \to \mathscr{N}(\Lambda)\) denote the function \(\mu \mapsto r_\Lambda(\mu)\), where, for any Borel set \(A\) of \(\Lambda\), \[r_\Lambda(\mu)(A) := \mu(A).\]

Under suitable conditions on a point process \(\Phi\) on $\R^N$ and a filtration function \(\kappa\), the support of the limiting measure \(\nu_q\) in the previous theorem is characterized in terms of the set \(R_q(\kappa)\) of $(\kappa,q)$-realizable points as follows.

\begin{theorem}[{\cite[Theorem~1.9]{hiraokaetal2018}}] \label{thm: measuresupport}\textcolor{white}{.}
    \begin{enumerate}[leftmargin=0em]
        \item[] \label{enumitem: unmarkedassumption}
        (Assumption) Let \(\Phi\) be a stationary point process on \(\R^N\).
        Assume that, for every compact set \(\Lambda \subset \R^N\), the distribution \(\Theta_{r_\Lambda(\Phi)}\) is absolutely continuous with respect to \(\Theta_{r_\Lambda(\Psi)}\), where \(\Psi\) is any  homogeneous Poisson point process on \(\R^N\) with intensity $\lambda\in (0,\infty)$, and that the Radon--Nikodym derivative \(d\Theta_{r_\Lambda(\Phi)} / d\Theta_{r_\Lambda(\Psi)}\) is strictly positive \(\Theta_{r_\Lambda(\Psi)}\)-almost surely.
        
        Further assume that the filtration function \(\kappa: \mathscr{F}(\R^N) \to [0, \infty)\) is $c_\kappa$-Lipschitz continuous with respect to the Hausdorff distance for some $c_\kappa>0$; that is, 
        for all \(\sigma, \sigma' \in \mathscr{F}(\R^N )\),
        \[
            |\kappa(\sigma) - \kappa(\sigma')|
            \leq c_\kappa d_H(\sigma, \sigma').
        \]
       
        \item[] \label{enumitem: conciseconclusion}
        (Conclusion) Then, for every \(q \geq 0\),
        \[
            \mathrm{supp}(\nu_q) = \overline{R_q(\kappa)},
        \]
        where the closure is taken in \(\Delta\).
    \end{enumerate}
\end{theorem}

We remark that the Lipschitz continuity assumption in Theorem~\ref{thm: measuresupport} is not a strong requirement. For example, both \(\kappa_C\) and \(\kappa_R\) satisfy this condition.

The next two theorems concern persistent Betti numbers.
The first establishes a strong law of large numbers, while the second provides a central limit theorem for the case of homogeneous Poisson point processes.

\begin{theorem}[{\cite[Theorem~1.11]{hiraokaetal2018}}] \label{thm: persbettiSLLN}
    Assume that \(\Phi\) is a stationary point process on \(\R^N\)
    and let \(\kappa\) be a filtration function.
    For \(n \geq 1\), let \(\Lambda_n = [-n/2, n/2)^N\).
    Then,
    for any \(0 \leq r \leq s < \infty\) and \(q \geq 0\), there exists a constant \(\widehat{\beta}^{\kappa, r, s}_q\) such that
    \[
        \frac{\mathbb{E}[\beta_{q, \Lambda_n}^{\kappa, r, s}]}{n^N} \to \widehat{\beta}^{\kappa, r,s}_q \quad \text{as } n \to \infty.
    \]
    In addition, if \(\Phi\) is ergodic, then
    \[
        \frac{\beta_{q, \Lambda_n}^{\kappa, r, s}}{n^N} \to \widehat{\beta}^{\kappa, r, s}_q \quad \text{almost surely as } n \to \infty.
    \]
\end{theorem}

Let \(\mathcal{N}(\mu, \sigma^2)\) denote the normal distribution with mean \(\mu\) and variance \(\sigma^2\), and let \(\xrightarrow{d}\) denote convergence in distribution of random variables \cite[p.~116]{durrett2019probability}.

\begin{theorem}[{\cite[Theorem~1.12]{hiraokaetal2018}}] \label{thm: persbettiCLT}
    Let \(\Phi\) be a homogeneous Poisson point process on \(\R^N\) with unit intensity, and let \(\kappa\) be a filtration function.
    For \(n \geq 1\), let \(\Lambda_n = [-n/2, n/2)^N\).
    Then, for any \(0 \leq r \leq s < \infty\) and \(q \geq 0\), there exists a constant \(\sigma_{r, s}^2 = \sigma_{r, s}^2(\kappa, q)\) such that
    \[
        \frac{\beta_{q, \Lambda_n}^{\kappa, r, s} - \mathbb{E}[\beta_{q, \Lambda_n}^{\kappa, r, s}]}{n^{N/2}} \xrightarrow{d} \mathcal{N}(0, \sigma_{r, s}^2) \quad \text{as } n \to \infty.
    \]
\end{theorem}

\begin{remark} \label{rmk: shirai}In \cite{shirai2022}, Theorems \ref{thm: vagueconv} and \ref{thm: persbettiSLLN} above are extended to the marked point process setting. Namely, \cite[Theorem~2.6]{shirai2022} extends Theorems \ref{thm: vagueconv} as follows.
 Let \(\Phi\) be a stationary \emph{marked} point process on \(\R^N\) with marks in \(\M\)
 and let \(\kappa: \mathscr{F}(\R^N \times \M) \to [0, \infty]\) be a filtration function.
    Also, let \(\mathcal{L} = \{\Lambda_n\}_n\) be \emph{any} convex averaging sequence in $\R^N$. Then, the same conclusion as Theorem~\ref{thm: vagueconv} holds.
        Furthermore, the limiting measure \(\nu_q\) in the conclusion is independent of the choice of \(\mathcal{L}\).
                
        \cite[Theorem~2.7]{shirai2022} extends Theorem~\ref{thm: persbettiSLLN} in a similar way.
\end{remark}

\section{Limit Theorems for Verbose Diagrams and Extended Persistent Betti Numbers} \label{sec: mainresults}

In this section, we extend the results in Section~\ref{sec: prevresults} to the setting of the verbose diagram (Theorems \ref{thm: vagueconvverbose}, \ref{thm: measuresupportverbose}, \ref{thm: persbettiSLLNextended}, and \ref{thm: persbettiCLTextended}). In addition, 
we also present extra new results regarding the total mass of the limiting measure and introduce the extended persistent Betti numbers (Theorem \ref{thm: limitmeasuretotalmass} and Definition~\ref{def: extendedPBN}). While the statements of the aforementioned theorems are presented in Section \ref{subsec:our results}, their proofs are deferred to Section \ref{subsec: pf}. 

\subsection{Our Results}\label{subsec:our results}

We extend Theorem~\ref{thm: vagueconv} and its marked analogue (Remark~\ref{rmk: shirai}) to the setting of verbose diagrams.

\begin{theorem}[Strong law of large numbers for verbose diagrams] \label{thm: vagueconvverbose}
    Assume that \(\Phi\) is a stationary marked point process on \(\R^N\) with marks in \(\M\),
    and let \(\kappa\) be a filtration function.
    Then, for each \(q \geq 0\), there exists a unique Radon measure \(\nu_q' = \nu_q'(\kappa,\Phi)\) on \(\Delta'\) such that, for any convex averaging sequence \(\mathcal{L} = \{\Lambda_n\}_n\) in \(\R^N\),
    \[
        \frac{1}{\vol{\Lambda_n}}\,\mathbb{E}[\xi_{\mathrm{Ver}, q, \Lambda_n}^\kappa] \xrightarrow{v} \nu_q'
        \quad\text{as } n \to \infty.
    \]
    Moreover, if \(\Phi\) is ergodic, then almost surely,
    \[
        \frac{1}{\vol{\Lambda_n}}\,\xi_{\mathrm{Ver}, q, \Lambda_n}^\kappa \xrightarrow{v} \nu_q'
        \quad\text{as } n \to \infty.
    \]
\end{theorem}

Next, we will see that, unlike \cite{hiraokaetal2018,shirai2022}, we can analyze the total mass of the limiting measure \(\nu_q'\), \emph{under the assumption that the filtration function \(\kappa\) is finite-valued}, i.e. $\kappa(\sigma)<\infty$ for all $\sigma\in \mathscr{F}(\R^N \times \M)$  (cf. the row including Theorem \ref{thm: limitmeasuretotalmass}  in Table \ref{tab: summaryofresults}); more specifically,
we will show that this assumption ensures the number of points in the verbose diagram is fully determined by the number of points in the input point cloud, unlike in the concise diagram (see \Cref{lem:verbose-cardinality}). This property will be instrumental in establishing Theorem~\ref{thm: limitmeasuretotalmass}.

\begin{theorem}[Total mass of the limiting verbose diagram] \label{thm: limitmeasuretotalmass}
In the setting of Theorem~\ref{thm: vagueconvverbose}, assume that $\Phi$ has intensity \(\lambda\in (0,\infty]\) and $\kappa$ is finite-valued.     
    Then, we have
    \[
        \nu_q'(\Delta') =
        \begin{cases}
            \lambda, & \text{if } q = 0, \\
            \infty, & \text{if } q \geq 1.
        \end{cases}
    \]
\end{theorem}
This theorem may be interpreted as follows. For \(q = 0\), (i) the order of growth of the number of points in the diagram coincides with (ii) the order of growth of the window's volume---making the total mass equal to the intensity.  
For \(q \geq 1\), (i) exceeds (ii), causing the total mass to diverge to infinity.

We remark that the filtration function $\kappa$ in the preceding theorem can, for example, be the filtration functions for the \v{C}ech or Vietoris--Rips filtrations (Definition~\ref{def:VR and Cech}).   

For any filtration function $\kappa$ for unmarked or marked point sets and integer $q\geq 0$, let us define \emph{$(\kappa,q)$-realizability} of a point $(b,d)\in\Delta'$ simply by, in Definition~\ref{def: realizable},  replacing $D_q(\K^\kappa(\Xi))$ and $\xi_{q}(\K^\kappa(\Xi))$ with $D_{\mathrm{Ver}, q}(\mathcal{C}(\K^\kappa(\Xi)))$ and $ \xi_{\mathrm{Ver}, q}(\mathcal{C}(\K^\kappa(\Xi)))$, respectively. We denote by \(R_q' = R_q'(\kappa) \subset \Delta'\) the set of all \((\kappa, q)\)-realizable points in \(\Delta'\).

Recall that \(\pi: \R^N \times \M \to \R^N\) denote the projection and that for a marked point process \(\Phi\), \(\Phi_g\) denotes its ground process.
We extend Theorem~\ref{thm: measuresupport} to:

\begin{theorem}[Support of the limiting verbose diagram] \label{thm: measuresupportverbose} \textcolor{white}{.}
    \begin{enumerate}[leftmargin=0em]
        \item[](Assumption) \label{enumitem: markedassumption}
         Let \(\Phi\) be a stationary marked point process on \(\R^N\) with marks in \(\M\).
        Assume that, for every compact set \(\Lambda \subset \R^N\), the distribution \(\Theta_{r_\Lambda(\Phi_g)}\) is absolutely continuous with respect to \(\Theta_{r_\Lambda(\Psi)}\), where \(\Psi\) is any homogeneous Poisson point process on \(\R^N\) with intensity $\lambda\in (0,\infty)$, 
        and that the Radon--Nikodym derivative \(d\Theta_{r_\Lambda(\Phi_g)} / d\Theta_{r_\Lambda(\Psi)}\) is strictly positive \(\Theta_{r_\Lambda(\Psi)}\)-almost surely.
        
        Further assume that a filtration function \(\kappa: \mathscr{F}(\R^N\times \M) \to [0,\infty)\) is Lipschitz continuous with respect to the Hausdorff distance on $\R^N$; that is, there exists a constant \(c_\kappa > 0\) such that for all \(\sigma, \sigma' \in \mathscr{F}(\R^N \times \M)\),
        \[
            |\kappa(\sigma) - \kappa(\sigma')|
            \leq c_\kappa d_H(\pi(\sigma), \pi(\sigma')).
        \]
        
        \item[](Conclusion) \label{enumitem: verboseconclusion}
         Then, for every \(q \geq 0\), the support of the limiting measure \(\nu_q'\) satisfies
        \[
            \mathrm{supp}(\nu_q') = \overline{R_q'(\kappa)},
        \]
        where the closure is taken in \(\Delta'\).
    \end{enumerate}
\end{theorem}

\begin{lemma} \label{lem: nomark}
    \begin{enumerate}[label=(\roman*)]
        \item \label{lemitem: concisesupport}
        The assumption of Theorem~\ref{thm: measuresupportverbose} implies the conclusion of Theorem~\ref{thm: measuresupport}.
        \item \label{lemitem: verbosesupport}
        If the assumption of Theorem~\ref{thm: measuresupport} implies the conclusion of Theorem~\ref{thm: measuresupportverbose}, then Theorem~\ref{thm: measuresupportverbose} holds.
        
    \end{enumerate}
\end{lemma}

\begin{proof}
    The proofs of the two statements are similar, and thus we only prove Item \ref{lemitem: concisesupport}.
    Suppose that the assumption of Theorem~\ref{thm: measuresupportverbose} holds.
    Then, by the Lipschitz continuity assumption on $\kappa$, for any \(\sigma \in \mathscr{F}(\R^N \times \M)\), the value \(\kappa(\sigma)\) is determined solely by its unmarked part \(\pi(\sigma)\), i.e., if \(\sigma' \in \mathscr{F}(\R^N \times \M)\) with \(\pi(\sigma) = \pi(\sigma')\), then \(\kappa(\sigma) = \kappa(\sigma')\).
        Accordingly, we define \(\widehat{\kappa}: \mathscr{F}(\R^N) \to [0,\infty]\) by
    \(
        \widehat{\kappa}(\sigma) := \kappa(\widetilde{\sigma}),
    \)
    where for each unmarked simplex \(\sigma \in \mathscr{F}(\R^N)\), the symbol \(\widetilde{\sigma}\) denotes \emph{any} marked simplex in \(\mathscr{F}(\R^N \times \M)\) with \(\pi(\widetilde{\sigma}) = \sigma\).
    It is straightforward to see that \(\widehat{\kappa}\) is a filtration function for unmarked point sets.
    We claim that
    \[
        \nu_q(\kappa,\Phi) = \nu_q(\widehat{\kappa},\Phi_g) 
        \quad \text{and} \quad 
        R_q(\kappa) = R_q(\widehat{\kappa}).
    \]
    This is because, for any
    simple marked point set \(\Xi\),
    the simplicial filtration $\K^{\kappa}(\Xi)$ is naturally isomorphic to $\K^{\widehat{\kappa}}(\pi(\Xi))$
    via the natural isomorphism \(\alpha:\K^{\kappa}(\Xi)\Rightarrow \K^{\widehat{\kappa}}(\pi(\Xi))\) given by
    \(\alpha_t: K^{\kappa}(\Xi,t) \to K^{\widehat{\kappa}}(\pi(\Xi),t)\), \(\sigma \mapsto \pi(\sigma)\) for all $t\geq 0$.
    
    Therefore, 
    by Theorem~\ref{thm: measuresupport},
    we have
    \[
        \mathrm{supp}(\nu_q(\kappa, \Phi)) = \mathrm{supp}(\nu_q(\widehat{\kappa},\Phi_g)) = \overline{R_q(\widehat{\kappa})} = \overline{R_q(\kappa)},
    \]as desired.
\end{proof}

\begin{example}
    Recall from Definition~\ref{def:VR and Cech} the filtration function $\kappa_C$. We find the set of realizable points $R_q'(\kappa_C)$.
    For \(q = 0\), it is easy to find that each of 
       \(
        R_0(\kappa_C)\) and \( 
        R_0'(\kappa_C)\) 
        is identical to  $\{0\} \times (0, \infty].
    $
    Also, for \(q \geq 1\), it is not difficult to verify that
    \begin{align*}
        R_q'(\kappa_C) &= R_q(\kappa_C) \cup \{(t, t): 0 < t < \infty\}.
    \end{align*}
    Hence, by Equation~\eqref{eq: cechrealizable}, for a stationary point process \(\Phi\) on \(\R^N\) that satisfies the assumption of Theorem~\ref{thm: measuresupportverbose}, we have
    \[
        \mathrm{supp}(\nu_q'(\kappa_C, \Phi)) = \overline{R_q'(\kappa_C)} = 
            \begin{cases}
        \{0\} \times [0, \infty], &\quad \text{if } q = 0, \\
        \Delta', &\quad \text{if } q = 1, 2, \dots, N-1, \\
        \{(t, t): 0 \leq t < \infty\}, &\quad \text{if } q \geq N.
    \end{cases}
    \]
\end{example}
From Equation~\eqref{eq: persbetti}, recall the definition of the persistent Betti number.  
Note that the right-hand side of that equation remains well-defined even when \(r > s\).  
This observation motivates an extension of the definition: 

\begin{definition} \label{def: extendedPBN}
    For any \(r, s \in [0, \infty)\) and a filtration \(\K\), the \emph{persistent Betti number} \(\beta_q^{r, s}(\K)\) is defined to be the right-hand side of Equation~\eqref{eq: persbetti}.
\end{definition}

\begin{convention}
    Given a measurable space \((Y, \mathcal{B})\) and a subset \(X\subset Y\), let \(\mu\) be a measure on the \(\sigma\)-algebra on \(X\) induced by \(\mathcal{B}\). Then, for any \(B \in \mathcal{B}\), we let \(
        \mu(B) := \mu(X \cap B).
    \)
\end{convention}

For the extended persistent Betti numbers, a representation analogous to Equation~\eqref{eq: persbettimeasure} remains valid, provided that the diagram is interpreted in its verbose form.

\begin{proposition} \label{prop: persbettimeasure}
    For \(r, s \in [0, \infty)\) and a filtration \(\K\) of a finite simplicial complex, we have
    \[
        \beta_q^{r, s}(\K) = \xi_{\mathrm{Ver}, q}(\mathcal{C}(\K))([0, r] \times (s, \infty]).
    \]
\end{proposition}

See Figure~\ref{subfig: PBNextended} for an illustration of the case \(r > s\).

\begin{remark}[Interpretation of the extended persistent Betti numbers] \label{rmk: PBNextended}
    While the phrase ``the number of \(q\)-dimensional holes that persist from time \(t = r\) to time \(t = s\)" does not make sense for $r>s$, the phrase ``the number of \(q\)-dimensional holes born no later than \(t=r\) and dying after \(t=s\)"---makes sense for all \(r, s \in [0, \infty)\).
    Proposition~\ref{prop: persbettimeasure} enables this interpretation for the extended persistent Betti numbers, provided that we even count holes that are born and die simultaneously.
    In particular, when \(r > s\), \(\beta_q^{r, s}\) equals the number of \(q\)-dimensional holes whose lifespan \([b, d]\) intersects the interval \((s, r]\).
\end{remark}

We extend Theorem~\ref{thm: persbettiSLLN} and its marked analogue (Remark~\ref{rmk: shirai}) to the setting of extended persistent Betti numbers.

\begin{theorem}[Strong law of large numbers for the extended persistent Betti number] \label{thm: persbettiSLLNextended}
    Assume that \(\Phi\) is a stationary marked point process on \(\R^N\) with marks in \(\M\),
    and let \(\kappa\) be a filtration function.  
    Then, for any \(r, s \in [0, \infty)\) and \(q \geq 0\), there exists a constant \(\widehat{\beta}^{\kappa, r,s}_q\) such that, for any convex averaging sequence \(\mathcal{L} = \{\Lambda_n\}_n\) in \(\R^N\),
    \[
        \frac{\mathbb{E}[\beta_{q, \Lambda_n}^{\kappa, r, s}]}{n^N} \to \widehat{\beta}^{\kappa, r,s}_q \quad \text{as } n \to \infty.
    \]
    Here, $\beta_{q, \Lambda_n}^{\kappa, r, s}$ stands for $\beta_q^{r, s}(\K^\kappa(\Phi_{\Lambda_n}))$.\footnote{We remark that $\beta_{q, \Lambda_n}^{\kappa, r, s}$ is a random variable, as in Theorems~\ref{thm: persbettiSLLN} and~\ref{thm: persbettiCLT}. We omit the proof of this remark (cf. Remark~\ref{rmk:point processes}).}
    In addition, if \(\Phi\) is ergodic, then
    \[
        \frac{\beta_{q, \Lambda_n}^{\kappa, r, s}}{n^N} \to \widehat{\beta}^{\kappa, r,s}_q \quad \text{almost surely as } n \to \infty.
    \]
\end{theorem}

The following theorem extends Theorem~\ref{thm: persbettiCLT} to the setting of extended persistent Betti numbers.

\begin{theorem}[Central limit theorem for the extended persistent Betti number] \label{thm: persbettiCLTextended}
    Let \(\Phi\) be a homogeneous Poisson point process on \(\R^N\) with unit intensity, and let \(\kappa\) be a filtration function.
    For \(n \geq 1\), let \(\Lambda_n = [-n/2, n/2)^N\).
    Then, for any \(r, s \in [0, \infty)\) and \(q \geq 0\), there exists a constant \(\sigma_{r, s}^2 = \sigma_{r, s}^2(\kappa, q)\) such that
    \begin{equation} \label{eq: persbettiCLTextended}
        \frac{\beta_{q, \Lambda_n}^{\kappa, r, s} - \mathbb{E}[\beta_{q, \Lambda_n}^{\kappa, r, s}]}{n^{N/2}} \xrightarrow{d} \mathcal{N}(0, \sigma_{r, s}^2) \quad \text{as } n \to \infty.
    \end{equation}
\end{theorem}

\subsection{Proofs} \label{subsec: pf}

We begin by introducing the method of `shifting the diagram' (described in the paragraph \textbf{Main ideas behind the proofs} from the introduction). This method will be used in the proofs of Theorems~\ref{thm: vagueconvverbose}, \ref{thm: persbettiSLLNextended}, and \ref{thm: persbettiCLTextended}.

\begin{definition}
    Let \(\kappa\) be a filtration function for an unmarked point set (resp. marked point sets).  
    For a real number \(t \geq 0\) and an integer \(k \geq 0\), a \emph{\(t\)-shift of \(\kappa\) at degree \(k\)} is the filtration function \(\kappa^{k, t}: \mathscr{F}(\R^N) \to [0, \infty]\) (resp. \(\kappa^{k, t}: \mathscr{F}(\R^N \times \M) \to [0, \infty]\)) defined by
    \[
        \kappa^{k, t}(\sigma) =
        \begin{cases}
            \kappa(\sigma), &\quad \text{if } \dim \sigma \leq k, \\
            \kappa(\sigma) + t, &\quad \text{otherwise}.
        \end{cases}
    \]
\end{definition}
We show that \(\kappa^{k, t}\) is indeed a filtration function.
    The assumption that \(\kappa\) is measurable directly implies that \(\kappa^{k,t}\) is also measurable.
    Conditions (K1) and (K2) (or (K2')) follow immediately.  
    Since $\kappa$ is a filtration function, there exists an increasing function \(\rho: [0, \infty] \to [0, \infty]\) 
    with \(\rho(t) < \infty\) for \(t < \infty\) such that
    \[
        \|x - y\| \leq \rho(\kappa(\{x, y\})).
    \]
    Since \(\kappa \leq \kappa^{k, t}\), we have
    \[
        \|x - y\| \leq \rho(\kappa(\{x, y\})) \leq \rho(\kappa^{k, t}(\{x, y\})),
    \]
    which shows that condition (K3) (or (K3')) holds.

We write \(D_{\mathrm{Ver}, q}^\kappa(X)\) for \(D_{\mathrm{Ver}, q}(\mathcal{C}(\K^\kappa(X)))\), and  \(\xi_{\mathrm{Ver}, q}^\kappa(X)\) for \(\xi_{\mathrm{Ver}, q}(\mathcal{C}(\K^\kappa(X)))\).
Also, for any multiset \(A\) of points in \(\overline{\R_{\geq 0}}^2\) and for any vector \(v \in \R^2\), let
\[
    A + v := \{a + v: a \in A\},
\]
where we set \(\infty + x := \infty\) for any $x\in \overline{\R_{\geq 0}}$.

\begin{lemma} \label{lem: shift}
    Let \(X\) be a finite subset of \(\R^N\) (resp. a simple marked point set), and let \(\kappa\) be a filtration function for unmarked point sets (resp. for marked point sets).  
    Then \(D_{\mathrm{Ver}, q}^{\kappa^{k, t}}(X)\) is obtained by translating the points of \(D_{\mathrm{Ver}, q}^\kappa(X)\) as follows:
    \[
        D_{\mathrm{Ver}, q}^{\kappa^{k, t}}(X) = 
        \begin{cases}
            D_{\mathrm{Ver}, q}^\kappa(X), &\quad \text{if } k > q, \\
            D_{\mathrm{Ver}, q}^\kappa(X) + (0, t), &\quad \text{if } k = q, \\
            D_{\mathrm{Ver}, q}^\kappa(X) + (t, t), &\quad \text{otherwise.}
        \end{cases}
    \]
\end{lemma}

\begin{figure}
    \centering
    \includegraphics[width=0.9\linewidth]{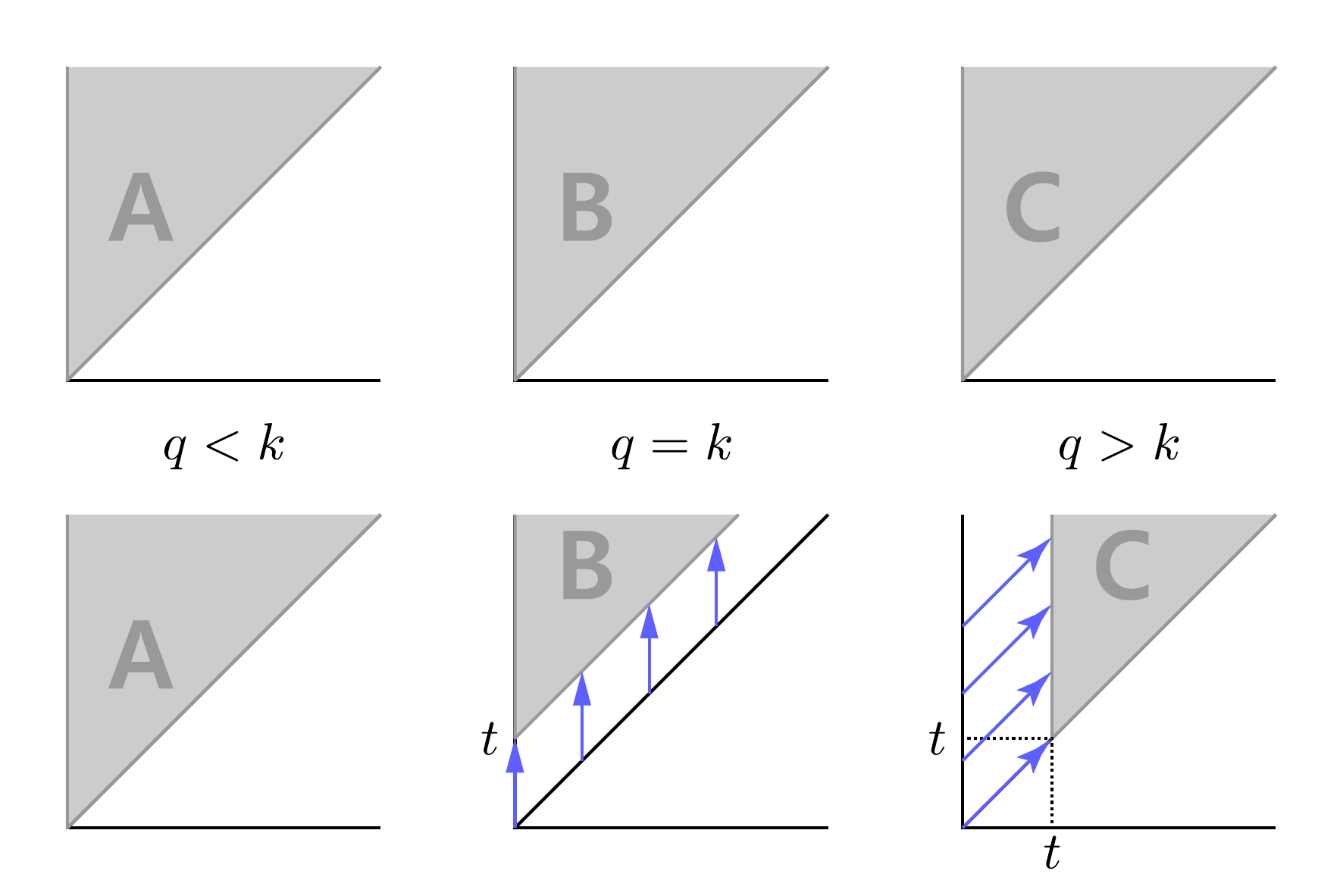}
    \caption{Illustration of Lemma \ref{lem: shift}.
    }
    \label{fig: t-shift}
\end{figure}
See Figure~\ref{fig: t-shift}. 

\begin{proof}
    Consider the FCC \(\mathcal{C}(\K^\kappa(X))\) defined as in Equation \eqref{eq: filtrationFCC}, which will be simply denoted by \((C_*, \partial, \ellK{\K^\kappa(X)})\). Then, we have \(\mathcal{C}(\K^{\kappa^{k, t}}(X))=(C_*, \partial, \ellK{\K^{\kappa^{k, t}}})\).
    For any integer \(q \geq 0\), let \(C_q\) be the \(q\)-chain subgroup of \(C_*\).
    Let \(x \in C_q\) for some \(q\), and we write  \(\dim x := q\).
    From Equations~\eqref{eq: vectfilt}~and~\eqref{eq: filtfuncandbirthtime}, we have:
    \begin{equation}
        \ellK{\K^{\kappa^{k, t}}(X)}(x) =
        \begin{cases}
            \ellK{\K^{\kappa}(X)}(x), &\quad \text{if } \dim x \leq k, \\
            \ellK{\K^{\kappa}(X)}(x) + t, &\quad \text{otherwise}.
        \end{cases}
        \label{eq: vffshift}
    \end{equation}
    Fix an integer \(q \geq 0\) and let \(\{x_1, \dots, x_r\}\) be an orthogonal set in \(C_q\) with respect to \(\ellK{\K^{\kappa}(X)}\).  
    Then, by Equation \eqref{eq: vffshift}, it is not difficult to see that \(\{x_1, \dots, x_r\}\) is also orthogonal with respect to \(\ellK{\K^{\kappa^{k, t}}(X)}\). 
    Therefore, a singular value decomposition of the boundary map \(\partial_{q + 1}: C_{q+1} \to \mathrm{Ker}\,\partial_q\) in \(\mathcal{C}(\K^\kappa(X))\) is also a singular value decomposition of \(\partial_{q + 1}\) in \(\mathcal{C}(\K^{\kappa^{k, t}}(X))\).
    Hence, \(D_{\mathrm{Ver}, q}^{\kappa^{k, t}}(X)\) is obtained by shifting the points in \(D_{\mathrm{Ver}, q}^\kappa(X)\) according to the rule in Equation~\eqref{eq: vffshift}.  
    This completes the proof.
\end{proof}

For any measure \(\mu\) on \(\overline{\R_{\geq 0}}^2\) and any \(v = (v_1, v_2) \in \R^2\) with \(v_1, v_2 \geq 0\), let \(\mu^{(v)}\) be the measure on \(\overline{\R_{\geq 0}}^2\) given by $A \mapsto \mu(A + v)$. For example, the equation in Lemma~\ref{lem: shift} is equivalent to:
    \begin{equation}\label{eq:translation}
       \xi_{\mathrm{Ver}, q}^\kappa (X) = 
        \begin{cases}
            \xi_{\mathrm{Ver}, q}^{\kappa^{k, t}}(X), &\quad \text{if } k > q, \\
           \xi_{\mathrm{Ver}, q}^{\kappa^{k, t}}(X)^{(0, t)}, &\quad \text{if } k = q, \\
            \xi_{\mathrm{Ver}, q}^{\kappa^{k, t}}(X)^{(t, t)}, &\quad \text{otherwise.}
        \end{cases}
    \end{equation}
The next lemma will be useful in the proof of Theorem~\ref{thm: vagueconvverbose}.

\begin{lemma}[Translations on \(\Delta'\)] \label{lem: topondelta'}
    For any real numbers \(u \geq t > 0\), we have:
    \begin{enumerate}[label=(\roman*)]
        \item (Preservation of compactness) A subset \(A\) of \(\Delta'\) is compact if and only if \(A + (t, u)\) is compact; \label{lemitem: compactondelta'}
        \item (Change of variables) For any measure \(\mu\) on \(\Delta'\) and a \(\mu^{((t, u))}\)-integrable function \(f: \Delta' \to \R\), we have
        \[
            \int_{\Delta'} f \, d\mu^{((t, u))} = \int_{\Delta' + (t, u)} f(x - (t, u)) \, d\mu(x).
        \] \label{lemitem: integrationondelta'}
    \end{enumerate}
\end{lemma}
Both statements appear natural; however, particular attention must be given to the ``infinity line'' $[0, \infty) \times \{\infty\}$ in $\Delta'$ and thus we include their proofs.

\begin{proof}
   Consider the subspace \(
        T' := \{(a, b) \in \R^2: -1 \leq a \leq b \leq 0\} \setminus \{(0, 0)\}
    \) of $\R^2$.
   Then, the map \(h: \Delta' \to T'\) given by
    \[
        h(a, b) =
        \begin{cases}
            (-e^{-a}, -e^{-b}), &\quad \text{if } b < \infty, \\
            (-e^{-a}, 0), &\quad \text{otherwise}
        \end{cases}
    \]
     is a homeomorphism. Also, for the linear map \(F: \R^2 \to \R^2\) given by
        \[(x_1, x_2)\mapsto
       (x_1, x_2) \begin{pmatrix}
            e^{-t} & 0 \\
            0 & e^{-u}
        \end{pmatrix},        
    \]
    we have
    \begin{equation} \label{eq: translationhomeo}
        h(x + (t, u)) = F(h(x)).
    \end{equation}

    We prove \Cref{lemitem: compactondelta'}. Since both $h$ and \(F\) are homeomorphisms, $A$ is compact iff
    \(F(h(A))\) is compact iff $h(A+(t,u))$ is compact iff $A+(t,u)$ is compact.
    
    \Cref{lemitem: integrationondelta'} can be shown using the standard change of variables for integration \cite[Proposition~2.6.8]{cohn2013measure} along with Equation~\eqref{eq: translationhomeo}:
    \begin{align*}
        \int_{\Delta'} f \, d\mu^{((t, u))} 
        &= \int_{\Delta'} (f \circ h^{-1} \circ h) \, d\mu^{((t, u))} \\
        &= \int_{T'} (f \circ h^{-1}) \, d(\mu^{((t, u))}  h^{-1}) \\
        &= \int_{T'} (f \circ h^{-1}) \, d(\mu[ (h^{-1} \circ F \circ h) \circ h^{-1}]) & \text{by \eqref{eq: translationhomeo}, applying \(h^{-1}\) on both sides} \\
        &= \int_{T'} (f \circ h^{-1}) \, d(\mu [h^{-1} \circ F]) \\
        &= \int_{F(T')} (f \circ h^{-1} \circ F^{-1}) \, d(\mu h^{-1}) \\
        &\stackrel{(\ast)}{=} \int_{F(T')} f(h^{-1}(y) - (t, u)) \, d(\mu h^{-1})(y)\\ 
        &= \int_{h(\Delta' + (t, u))} f(h^{-1}(y) - (t, u)) \, d(\mu h^{-1})(y) & \text{by \eqref{eq: translationhomeo}, since \(T' = h(\Delta')\)} \\
        &= \int_{\Delta' + (t, u)} f(x - (t, u)) \, d\mu(x),
    \end{align*}
    where equation $(\ast)$ follows by setting \(x = h^{-1}(y) - (t, u)\) in Equation \eqref{eq: translationhomeo} and applying $h^{-1}\circ F^{-1}$ to the both sides.
\end{proof}

\begin{proof}[Proof of Theorem~\ref{thm: vagueconvverbose}]
    Fix \(t > 0\). 
    By Theorem~\ref{thm: vagueconv}, there exists a Radon measure \(\nu_q\) on \(\Delta\) such that 
    \begin{equation} \label{eq: eta_n convergence}
        \eta_n := \frac{1}{\vol{\Lambda_n}} \mathbb{E}[\xi_{q, \Lambda_n}^{\kappa^{q, t}}] \xrightarrow{v} \nu_q \quad \text{as } n \to \infty.
    \end{equation}
    For the first part of the theorem, we will show that (i)
    \(\nu_q'\) is a Radon measure, and that (ii)  
     \(\frac{1}{\vol{\Lambda_n}} \mathbb{E}[\xi_{\mathrm{Ver}, q, \Lambda_n}^\kappa] \xrightarrow{v} \nu_q':= \nu_q^{((0, t))}\) as \(n \to \infty\).
    Once these are proved, the sequence cannot converge to a Radon measure other than $\nu_q'$ in  \(\mathscr{{R}}(\Delta')\) since the space \(\mathscr{{R}}(\Delta')\) is metrizable \cite[Theorem~4.2]{kallenberg2017}.  
    
    For Item (i), let \(F\) be a compact subset of \(\Delta'\).  
    By Lemma~\ref{lem: topondelta'}~\ref{lemitem: compactondelta'}, the set \(F + (0, t)\) is compact (and contained in \(\Delta\)).  
    Since \(\nu_q\) is a Radon measure on \(\Delta\), we have
    \(
        \nu_q'(F) = \nu_q(F + (0, t)) < \infty,
    \) as desired.  Next we prove Item (ii).
    
    For \(n \geq 1\), we have 
    \begin{equation} \label{eq: shiftofverbosediagram}
    \begin{aligned} 
        \xi_{\mathrm{Ver}, q, \Lambda_n}^\kappa &= (\xi_{\mathrm{Ver}, q, \Lambda_n}^{\kappa^{q, t}})^{((0, t))}&\mbox{by \Cref{eq:translation}}  \\&= (\xi_{q, \Lambda_n}^{\kappa^{q, t}})^{((0, t))},
     \end{aligned}   
    \end{equation}
    because, for the filtration \(\kappa^{q, t}\),
    the degree-$q$ verbose diagram $\xi_{\mathrm{Ver}, q, \Lambda_n}^{\kappa^{q, t}}$ and the degree-$q$ concise diagram $\xi_{q, \Lambda_n}^{\kappa^{q, t}}$ coincide: Recall that the concise diagram is obtained by deleting all points on the diagonal from the verbose diagram.
     Every point in \(\xi_{\mathrm{Ver}, q, \Lambda_n}^{\kappa^{q, t}}\) is a translation of a point in \(\xi_{\mathrm{Ver}, q, \Lambda_n}^\kappa\) by \((0, t)\), and thus none of them lies on the diagonal, which makes the verbose and concise diagram coincide. 
    
    Moreover, by \Cref{eq:expectation,eq: shiftofverbosediagram}, for any Borel set \(A \subset \Delta'\), we have  
    \begin{align*}
        \mathbb{E}[\xi_{\mathrm{Ver}, q, \Lambda_n}^\kappa](A) 
        &= \mathbb{E}[\xi_{\mathrm{Ver}, q, \Lambda_n}^\kappa(A)] & \text{by Equation~\eqref{eq:expectation}} \\
        &= \mathbb{E}[(\xi_{q, \Lambda_n}^{\kappa^{q, t}})^{((0, t))}(A)] & \text{by Equation~\eqref{eq: shiftofverbosediagram}} \\
        &= \mathbb{E}[\xi_{q, \Lambda_n}^{\kappa^{q, t}}(A + (0, t))] & \text{by definition} \\
        &= \mathbb{E}[\xi_{q, \Lambda_n}^{\kappa^{q, t}}](A + (0, t)) & \text{by Equation~\eqref{eq:expectation}} \\
        &= \mathbb{E}[\xi_{q, \Lambda_n}^{\kappa^{q, t}}]^{((0, t))}(A). & \text{by definition}
    \end{align*}
        Hence, for \(\mu_n := \frac{1}{\vol{\Lambda_n}} \mathbb{E}[\xi_{\mathrm{Ver}, q, \Lambda_n}^\kappa]\),
        we have
    \begin{equation} \label{eq: munu}
        \mu_n = \eta_n^{((0, t))}.
    \end{equation}
    Let \(f : \Delta' \to \R\) be a continuous function with compact support.  
    It suffices to
    show that
    \begin{equation} \label{eq: convergence}
        \lim_{n \to \infty} \int_{\Delta'} f\, d\mu_n = \int_{\Delta'} f\, d\nu_q'.
    \end{equation}
    By {Equation~\eqref{eq: munu}} and Lemma~\ref{lem: topondelta'}~\ref{lemitem: integrationondelta'}, we have
    \begin{align} \label{eq: changeofvariables1}
        \begin{split}
            \lim_{n \to \infty} \int_{\Delta'} f \, d\mu_n 
            &= \lim_{n \to \infty} \int_{\Delta'} f \, d\eta_n^{((0, t))} \\
            &= \lim_{n \to \infty} \int_{\Delta' + (0, t)} f(x - (0, t)) \, d\eta_n(x).
        \end{split}
    \end{align}

Let $g:\Delta\rightarrow \R$ be any extension of the function $f(\bullet-(0,t)):\Delta'+(0,t)\rightarrow \R$, which is continuous and compactly supported. One possible choice of $g$ is defined by linear interpolation between zero and the values of $f$ on the line $y=x+t$  within the region \(\{(x,y)\in \Delta:x + t/2 \leq y < x + t\}\).
    Then, for each \(n \geq 1\), since \(\eta_n\) and \(\nu_q\) vanish below the line \(y = x + t\), we have
    \begin{align*}
        \int_{\Delta' + (0, t)} f(x - (0, t))\, d\eta_n(x) &= \int_\Delta g(x)\, d\eta_n(x), \text{ and}\\
        \int_{\Delta' + (0, t)} f(x - (0, t))\, d\nu_q(x) &= \int_\Delta g(x)\, d\nu_q(x).
    \end{align*}
    These equations imply the first and third equalities below:
    \begin{equation} \label{eq: changeofvariables2}
        \begin{aligned}
            \lim_{n \to \infty} \int_{\Delta' + (0, t)} f(x - (0, t))\, d\eta_n(x)
            &= \lim_{n \to \infty} \int_\Delta g(x)\, d\eta_n(x) \\
            &= \int_\Delta g(x)\, d\nu_q(x) 
            & \text{by Equation~\eqref{eq: eta_n convergence}} \\
            &= \int_{\Delta' + (0, t)} f(x - (0, t))\, d\nu_q(x) \\
            &= \int_{\Delta'} f\, d(\nu_q)^{((0, t))}. 
            & \text{by Lemma~\ref{lem: topondelta'}~\ref{lemitem: integrationondelta'}}
        \end{aligned}
    \end{equation}
    {Now Equation~\eqref{eq: convergence} follows from Equations~\eqref{eq: changeofvariables1}~and~\eqref{eq: changeofvariables2}, completing the proof of Item (ii).} 
   {This ends the proof for the first part of the theorem.}
    
    Next, assume that $\Phi$ is ergodic. We  again exploit Theorem~\ref{thm: vagueconv}, especially its second part.   In the preceding proof,  simply 
    redefine \(\eta_n := \frac{1}{\vol{\Lambda_n}} \xi_{q, \Lambda_n}^{\kappa^{q, t}}\) and \(\mu_n := \frac{1}{\vol{\Lambda_n}} \xi_{\mathrm{Ver}, q, \Lambda_n}^\kappa\), 
    remove every \(\mathbb{E}\), and assume that all convergences hold almost surely.
    The result follows analogously, since \(\eta_n \xrightarrow{v} \nu_q\) almost surely.
\end{proof}

{Next, we establish a connection between the points in a verbose diagram and the generators of the cycle and boundary groups of the filtration that the diagram is derived from:
Let \(\K = \{K_t\}_{t \in [0, \infty)}\) be a filtration of a finite simplicial complex \(K\) with \(\bigcup_{t \geq 0} K_t = K\).
Then, consider the FCC \(\mathcal{C}(\K) = (C_*(K), \partial, \ellK{\K})\), where \(C_*(K) = \bigoplus_{q = 0}^\infty C_q\), and for each $q\geq 0$, \(\partial_{q+1}: C_{q+1} \to \mathrm{Ker}\, \partial_q\) denotes the restriction of \(\partial\), which is a linear map.}

Fix \(q \geq 0\) and let \(S = ((y_1, \dots, y_n), (x_1, \dots, x_m))\) be a singular value decomposition of  {\(\partial_{q+1}\)}.  
Then, by the definition of the verbose diagram, we have a bijection \(\{x_1, \dots, x_m\} \to D_{\mathrm{Ver}, q}(\mathcal{C}(\K))\) defined by
\[
    {x_i \mapsto}
    \begin{cases}
        (\ellK{\K}(x_i), \ellK{\K}(y_i)), &\quad \text{if } i \leq \mathrm{rank}(\partial_{q+1}), \\
        (\ellK{\K}(x_i), \infty), &\quad \text{otherwise.}
    \end{cases}
\]
Let \(c_S: D_{\mathrm{Ver}, q}(\mathcal{C}(\K)) \to \{x_1, \dots, x_m\}\) be the inverse of this map.

\begin{proposition}\label{prop: cycleboundary}
    Let \(\K\), \(q\), \(S = ((y_1, \dots, y_n), (x_1, \dots, x_m))\), \(c_S\) be as in the preceding paragraph.
    Then, for any \(t \in [0, \infty)\),
    \begin{enumerate}[label=(\roman*)]
        \item the set \(P := \{c_S((b, d)) : (b, d) \in D_{\mathrm{Ver}, q}(\mathcal{C}(\K)),\ b \leq t\}\) is a basis for the cycle group \(\mathcal{Z}_q(K_t)\);
        \item the set \(Q := \{c_S((b, d)) : (b, d) \in D_{\mathrm{Ver}, q}(\mathcal{C}(\K)),\ d \leq t\}\) is a basis for the boundary group \(\mathcal{B}_q(K_t)\).
    \end{enumerate}
\end{proposition} 

\begin{corollary} \label{cor: cycleboundarymeasure}
    Let \(\K\), \(q\) be as in Proposition~\ref{prop: cycleboundary}.
    Then, for any \(t \in [0, \infty)\),
    \begin{enumerate}[(i)]
        \item \(\xi_{\mathrm{Ver}, q}(\mathcal{C}(\K))
        ([0, t] \times [0, \infty])
        = \dim \mathcal{Z}_q(K_t)\); \label{item:cycleboundary1}
        \item \(\xi_{\mathrm{Ver}, q}(\mathcal{C}(\K))
        ([0, \infty] \times [0, t])
        = \dim \mathcal{B}_q(K_t)\).\label{item:cycleboundary2}
    \end{enumerate}
\end{corollary}

For any finite set \(A\), let \(|A|\) denote its cardinality.

\begin{proof}
    Since \(c_S\) in Proposition~\ref{prop: cycleboundary} is a bijection, we have that \(|P| = \xi_{\mathrm{Ver}, q}(\mathcal{C}(\K))([0, t] \times [0, \infty])\) and \(|Q| = \xi_{\mathrm{Ver}, q}(\mathcal{C}(\K))([0, \infty] \times [0, t])\).
    Then, both claims follow from Proposition~\ref{prop: cycleboundary}.
\end{proof}

\begin{proof}[Proof of Proposition~\ref{prop: cycleboundary}]
    We first note that $P=\{x_i:\ell_{\K}(x_i)\leq t\}$ and $Q=\{x_i: \ell_{\K}(y_i)\leq t\}$.
    Both \(P\) and \(Q\) are linearly independent, as they are subsets of a basis for {\(\mathrm{Ker}\, \partial_q\).}
    We have \(P \subset \mathcal{Z}_q(K_t)\): if \(x_i \in P\), then \(\ellK{\K}(x_i) \leq t\), and all simplices appearing with nonzero coefficients in \(x_i\) belong to \(K_t\).
    Similarly, \(Q \subset \mathcal{B}_q(K_t)\): if \(x_i \in Q\), then \(\ellK{\K}(y_i) \leq t\), and all simplices appearing with nonzero coefficients in \(y_i\) belong to \(K_t\), so that \(x_i = \partial_{q+1} y_i\) is a boundary in \(K_t\).
    Thus, it remains to show that \(P\) spans \(\mathcal{Z}_q(K_t)\), and \(Q\) spans \(\mathcal{B}_q(K_t)\).

    Suppose that \(z \in \mathcal{Z}_q(K_t) \subset \mathrm{Ker}\, \partial_q\), and we show that $z$ is a linear combination of elements in $P$.  
    By Remark~\ref{rmk:meaning of ell}, we have that \(\ellK{\K}(z)\leq t\).
    Since $\{x_1,\ldots,x_m\}$ is a basis for $\mathrm{Ker}\, \partial_q$, there exist $z_i\in \mathbb{F}$, $i=1,\ldots,m$ such that  \(z=\sum_{i=1}^m z_i x_i\). It suffices to show that if \(\ellK{\K}(x_j) > t\) for some \(j\), then $z_j=0$. 
    Assume that \(\ellK{\K}(x_j) > t\) for some \(j\).  
    Then, since the \(x_i\)'s are orthogonal and \(\ellK{\K}(z) \leq t\),  we have that \(z_j = 0\).  

    Next, let \(w \in \mathcal{B}_q(K_t)\), and we show that $w$ is a linear combination of elements in $Q$.
    Pick any \(v \in C_{q+1}(K_t) \subset C_{q+1}\) such that \(\partial_{q+1} v = w\).
    By Remark~\ref{rmk:meaning of ell}, we have that \(\ellK{\K}(v)\leq t\).
    Since $\{y_1,\ldots,y_n\}$ is a basis for $C_{q+1}$, there exist $v_i\in \mathbb{F}$, $i=1,\ldots,n$ such that  \(v=\sum_{i=1}^n v_i y_i\).
    By the same reasoning as in the previous paragraph,
    we have \(\ellK{\K}(y_i) \leq t\) if \(v_i \neq 0\).
    Since \(w = \partial_{q+1}\left(\sum_{i=1}^n v_i y_i\right)=\sum_{i=1}^n v_i\partial_{q+1}(y_i)\), and each $\partial_{q+1}(y_i)$ is either $x_i$ or $0$, it follows that \(w\) is a linear combination of elements in \(Q\).  
\end{proof}

\begin{lemma}\label{lem:verbose-cardinality} Let \(\kappa\) be a finite-valued filtration function for marked point sets (resp. for unmarked point sets) and let \(X\) be a simple marked point set (resp. finite subset of \(\R^N\)). 
Then,
    \[
        \xi_{\mathrm{Ver}, q}^\kappa(X)(\Delta') =
        \begin{cases}
            |X|, &\quad \text{if } q = 0, \\
            \binom{|X| - 1}{q + 1}, &\quad \text{if } 1 \leq q \leq |X| - 2, \\
            0, &\quad \text{if } q \geq |X| - 1.
        \end{cases}
    \]
\end{lemma}
This lemma is a direction extension of  \cite[Example~3.12]{mémolizhou2024}), which only concerns the Vietoris-Rips filtration. The proof is almost the same, and thus we omit it.

\begin{proof}[Proof of Theorem~\ref{thm: limitmeasuretotalmass}]
   Since the limiting measure $\nu_q'$ in Theorem \ref{thm: vagueconvverbose} does not depend on the choice of the convex averaging sequence $\mathcal{L}$, we will assume \(\mathcal{L}=\{\Lambda_n = [-n/2, n/2)^N\}_n\).

\paragraph{Case of $q=0$.}
    For \(q = 0\), we first give a simpler proof assuming \(\M = \{*\}\), i.e. the unmarked setting.
    Suppose \(\M = \{*\}\). Then, by the translation invariance of \(\kappa\), all zero-simplices have the same \(\kappa\)-value, say  
     \(\alpha\).  
    Then for each \(n \geq 1\), \(\xi_{\mathrm{Ver}, 0, \Lambda_n}^\kappa\) is supported on a subset of \(A := \{\alpha\} \times [\alpha, \infty]\), and thus so is \(\mathbb{E}[\xi_{\mathrm{Ver}, 0, \Lambda_n}^\kappa]\).

    Note that \(A\) is compact
    because the image of $A$ via the homeomorphism \(h\) in the proof of Lemma~\ref{lem: topondelta'} is compact.
    Let \(B\) be any open neighborhood of \(A\) with compact closure.
    Since the support of  \(\xi_{\mathrm{Ver}, 0, \Lambda_n}^\kappa\) is in $A$,   
    we have \(\xi_{\mathrm{Ver}, 0, \Lambda_n}^\kappa(A) = \xi_{\mathrm{Ver}, 0, \Lambda_n}^\kappa(B)\), and
    by Lemma~\ref{lem:verbose-cardinality},
    their common value equals the number of points in \(\Phi_{\Lambda_n}\), which is identical with the value \(\Phi(\Lambda_n)\).
    Hence, we have
    \[
        \mathbb{E}[\xi_{\mathrm{Ver}, 0, \Lambda_n}^\kappa](A) = \mathbb{E}[\xi_{\mathrm{Ver}, 0, \Lambda_n}^\kappa](B) = \mathbb{E}[\Phi(\Lambda_n)].
    \]
    Since \(\Phi\) is stationary,
    the linearity of expectation implies that   
    \begin{equation}\label{eq:linearity}
        \mathbb{E}[\Phi(\Lambda_n)] = \lambda n^N
    \end{equation}    
     and thus
    \[
        \frac{\mathbb{E}[\Phi(\Lambda_n)]}{\vol{\Lambda_n}} = \frac{\mathbb{E}[\Phi(\Lambda_n)]}{n^N} = \lambda.
    \]
    By Theorem~\ref{thm: vagueconvverbose}, we have
    \begin{equation}\label{eq:vague convergence}
        \frac{1}{\vol{\Lambda_n}} \mathbb{E}[\xi_{\mathrm{Ver}, 0, \Lambda_n}^\kappa] \xrightarrow{v} \nu_0'.
    \end{equation}
    Then, by Lemma~\ref{lem: vagueconvequiv}, we have that
    \[
         \lambda = \limsup_{n \to \infty} \frac{\mathbb{E}[\Phi(\Lambda_n)]}{\vol{\Lambda_n}} \leq \nu_0'(A) \quad \text{and} \quad
        \nu_0'(B) \leq \liminf_{n \to \infty} \frac{\mathbb{E}[\Phi(\Lambda_n)]}{\vol{\Lambda_n}} = \lambda.
    \]
    Therefore, \(\nu_0'(A) = \nu_0'(B) = \lambda\). Note that $\Delta'$ is a union of an increasing sequence of open neighborhoods of $A$ with compact closures. Hence, by continuity  from below of measures, $\nu_0'(\Delta')=\lambda$.
    
    Next, we present a proof \emph{without} assuming that $\M=\{\ast\}$. 
    Now, zero-simplices can have different birth-times due to  marks.
    For any \(a \in \R_{>0}\), let
    \begin{equation} \label{eq: AaandBa}
            A_{a} := \{(x, y) \in \Delta': x \leq a\} \mbox{ and }
            B_{a} := \{(x, y) \in \Delta': x < a\}.
    \end{equation}
    Since \(\bigcup_{a = 1}^\infty A_a = \Delta'\) and \(\bigcup_{a = 1}^\infty B_a = \Delta'\),
    the continuity from below of measures implies
    \begin{equation} \label{eq: totalmasszero}
        \nu_0'(\Delta') = \sup_{a = 1, 2, \dots} \nu_0'(A_a) = \sup_{a = 1, 2, \dots} \nu_0'(B_a).
    \end{equation}

    First, we show that \(\nu_0'(\Delta') \geq \lambda\).
    Fix \(a > 0\).
    Since \(A_a\) is compact, Lemma~\ref{lem: vagueconvequiv} and the vague convergence in Equation~\eqref{eq:vague convergence} imply that 
    \begin{equation} \label{eq: limsup}
         \limsup_{n \to \infty} \frac{1}{n^N}\mathbb{E}[\xi_{\mathrm{Ver}, 0, \Lambda_n}^\kappa(A_a)] \leq \nu_0'(A_a).
    \end{equation}
    Also, Corollary~\ref{cor: cycleboundarymeasure} \ref{item:cycleboundary1} implies that, for any \(n\in \mathbb{Z}_{>0}\),
    \[
        \xi_{\mathrm{Ver}, 0, \Lambda_n}^\kappa(A_a) = \dim \mathcal{Z}_0(K^\kappa(\Phi_{\Lambda_n}, a)) \quad \text{(cf. Equation~\eqref{eq: sublevel}).}
    \]
    Note that the vertex set of the simplicial complex $K^\kappa(\Phi_{\Lambda_n}, a)$ is $\{p \in \Phi_{\Lambda_n}: \kappa(\{p\}) \leq a\}$, and thus $\{\{p\} : p \in \Phi_{\Lambda_n} \mbox{ and } \kappa(\{p\}) \leq a\}$ is a basis for the $0$-chain group of $K^\kappa(\Phi_{\Lambda_n}, a)$. Since every \(0\)-chain  is a cycle, we have
    \[
        \dim \mathcal{Z}_0(K^\kappa(\Phi_{\Lambda_n}, a)) = |\{p \in \Phi_{\Lambda_n}: \kappa(\{p\}) \leq a\}|.
    \]
By the translation invariance of \(\kappa\) (Definition~\ref{def:filtration function} (K2')),  the function \(\widetilde{\kappa}: \M \to [0, \infty)\) given by \(
        \widetilde{\kappa}(m) := \kappa(\{(x, m)\})  \ \ 
        \mbox{for any \(m \in \M\) and any \(x \in \R^N\)}       
    \)  is well-defined. Then, we have
    \[
        |\{p \in \Phi_{\Lambda_n}: \kappa(\{p\}) \leq a\}| = \Phi(\Lambda_n \times \widetilde{\kappa}^{-1}([0, a])).
    \]
    The last three equations imply
    \begin{equation} \label{eq: Aaandkappatilde}
        \xi_{\mathrm{Ver}, 0, \Lambda_n}^\kappa(A_a) = \Phi(\Lambda_n \times \widetilde{\kappa}^{-1}([0, a])).
    \end{equation}
    Taking expectation on both sides, the translation invariance of \(\Phi\) gives
    \[
        \mathbb{E}[\xi_{\mathrm{Ver}, 0, \Lambda_n}^\kappa(A_a)] = \mathbb{E}[\Phi](\Lambda_n \times \widetilde{\kappa}^{-1}([0, a])) = n^N \mathbb{E}[\Phi](\Lambda_1 \times \widetilde{\kappa}^{-1}([0, a])),
    \]
    and thus
    \[
        \frac{1}{n^N}\mathbb{E}[\xi_{\mathrm{Ver}, 0, \Lambda_n}^\kappa(A_a)] = \mathbb{E}[\Phi](\Lambda_1 \times \widetilde{\kappa}^{-1}([0, a])).
    \]
    Since the right-hand side is independent of \(n\),
    Equation~\eqref{eq: limsup} implies that
    \[
          \mathbb{E}[\Phi](\Lambda_1 \times \widetilde{\kappa}^{-1}([0, a])) \leq \nu_0'(A_a).
    \]
    Since this holds for any \(a > 0\), Equation~\eqref{eq: totalmasszero} and the continuity from below of measures give
    \begin{align*}
        \nu_0'(\Delta') &= \sup_{a = 1, 2, \dots} \nu_0'(A_a) \\
        &\geq \sup_{a = 1, 2, \dots} \mathbb{E}[\Phi](\Lambda_1 \times \widetilde{\kappa}^{-1}([0, a])) \\
        &= \mathbb{E}[\Phi](\Lambda_1 \times \M) \\
        &= \lambda 
    \end{align*}
    as desired.
    
    Next, we show that \(\nu_0'(\Delta') \leq \lambda\).
    Again, fix \(a > 0\).
    Note that \(B_a\) is relatively compact since its closure is the compact set $A_a$.
    Thus, Lemma~\ref{lem: vagueconvequiv} and Theorem~\ref{thm: vagueconvverbose} imply that
    \begin{equation} \label{eq: liminf}
        \nu_0'(B_a) \leq \liminf_{n \to \infty} \frac{1}{n^N}\mathbb{E}[\xi_{\mathrm{Ver}, 0, \Lambda_n}^\kappa(B_a)].
    \end{equation}
    Let \(M > 0\) be an integer with \(a > 1/M\). Then, we have \(\bigcup_{i = M}^\infty A_{a - 1/i} = B_a\).
    Hence, by Equation~\eqref{eq: Aaandkappatilde} and the continuity from below of measures, we have
    \begin{equation} \label{eq: Baandkappatilde}
        \begin{split}
            \xi_{\mathrm{Ver}, 0, \Lambda_n}^\kappa(B_a) &= \sup_{i= M, M+1, \dots} \xi_{\mathrm{Ver}, 0, \Lambda_n}^\kappa(A_{a - 1/i})\\
            &= \sup_{i= M, M+1, \dots} \Phi(\Lambda_n \times \widetilde{\kappa}^{-1}([0, a-1/i])) \\
            &= \Phi(\Lambda_n \times \widetilde{\kappa}^{-1}([0, a))).
        \end{split}
    \end{equation}
    After taking the expectation on both sides of the equality $\xi_{\mathrm{Ver}, 0, \Lambda_n}^\kappa(B_a)=\Phi(\Lambda_n \times \widetilde{\kappa}^{-1}([0, a)))$, a similar argument as after Equation~\eqref{eq: Aaandkappatilde} yields 
     \[
        \frac{1}{n^N}\mathbb{E}[\xi_{\mathrm{Ver}, 0, \Lambda_n}^\kappa(B_a)] = \mathbb{E}[\Phi](\Lambda_1 \times \widetilde{\kappa}^{-1}([0, a)))
    \]
    and in turn
    \[
        \nu_0'(B_a) \leq \mathbb{E}[\Phi](\Lambda_1 \times \widetilde{\kappa}^{-1}([0, a))).
    \]
    Since this holds for any \(a > 0\), Equation~\eqref{eq: totalmasszero} and the continuity from below of measures give
    \begin{align*}
        \nu_0'(\Delta') &= \sup_{a = 1, 2, \dots} \nu_0'(B_a) \\
        &\leq \sup_{a = 1, 2, \dots} \mathbb{E}[\Phi](\Lambda_1 \times \widetilde{\kappa}^{-1}([0, a))) \\
        &= \mathbb{E}[\Phi](\Lambda_1 \times \M) \\
        &= \lambda
    \end{align*}
    as desired.
    This completes the proof for \(q = 0\).

\paragraph{Case of $q\geq 1$.}
We wish to show that $\nu_q'(\Delta')=\infty$.  To this end, we will utilize the following claims, which will be proved at the end of the proof.    

\noindent\textit{Claim 1.} If the intensity \(\lambda\) is finite, then    \(
        \frac{1}{n^N} \mathbb{E}[\xi_{\mathrm{Ver}, q, \Lambda_n}^\kappa](\Delta') \to \infty \quad \text{as } n \to \infty.
    \)

\noindent\textit{Claim 2.} If the intensity \(\lambda\) is infinite, then for each $n\geq 1$, \(
        \frac{1}{n^N} \mathbb{E}[\xi_{\mathrm{Ver}, q, \Lambda_n}^\kappa](\Delta') 
    =\infty\). 

    Fix \(M > 0\).  
    For each \(n \geq 1\), let \(\mu_n := \frac{1}{n^N} \mathbb{E}[\xi_{\mathrm{Ver}, q, \Lambda_n}^\kappa]\).  
    Whether \(\lambda\) is finite or infinite, by the claims there exists \(n_0 \geq 1\) such that \(\mu_{n_0}(\Delta') > M\).
    Recall the set \(A_a\) from Equation~\eqref{eq: AaandBa}  and that \(\bigcup_{a = 1}^\infty A_a = \Delta'\).
    Using the continuity from below of measures, we conclude that there exists \(x_0 > 0\) such that
    \[
        \mu_{n_0}(A_{x_0}) > M.
    \]
    Define \(n_k := 2^k n_0\) for \(k \geq 0\).
    The sequence \(\{\mu_{n_k}(A_{x_0})\}_k\) is non-decreasing, as will be shown later.
    Hence, the compactness of \(A_{x_0}\),
    Lemma~\ref{lem: vagueconvequiv}, and Theorem~\ref{thm: vagueconvverbose} imply that
    \[
         M < \lim_{k \to \infty}\mu_{n_k}(A_{x_0}) \leq \limsup_{n \to \infty} \mu_n(A_{x_0}) \leq \nu_q'(A_{x_0}) \leq \nu_q'(\Delta'). 
    \]
    Since \(M > 0\) was arbitrary, it follows that \(\nu_q'(\Delta') = \infty\).

    It remains to show that the sequence \(\{\mu_{n_k}(A_{x_0})\}_k\) is non-decreasing.  
    Note that  by Corollary~\ref{cor: cycleboundarymeasure}~\ref{item:cycleboundary1}, for each \(n \geq 1\),  \(\xi_{\mathrm{Ver}, q, \Lambda_n}^\kappa(A_{x_0}) = \dim \mathcal{Z}_q(K^\kappa(\Phi_{\Lambda_n}, x_0))\).
    Observe that the box \(\Lambda_{n_1}\) is partitioned into \(2^N\) boxes congruent to \(\Lambda_{n_0}\).  
    Denote these boxes by \(Q_1, \dots, Q_{2^N}\).  
    For \(i=1,\ldots,2^N\), let \(V_i := \mathcal{Z}_q(K^\kappa(\Phi_{Q_i}, x_0))\).
    Since the boxes \(Q_1, \dots, Q_{2^N}\) are disjoint, the subspaces \(V_1, \dots, V_{2^N}\) mutually trivially intersect, and their direct sum
    is a subspace of \(\mathcal{Z}_q(K^\kappa(\Phi_{\Lambda_{n_1}}, x_0))\).  
    Therefore,
    \[
        \sum_{i = 1}^{2^N} \dim V_i \leq \dim \mathcal{Z}_q(K^\kappa(\Phi_{\Lambda_{n_1}}, x_0)).
    \]
    Taking expectations and applying translation invariance along with Corollary~\ref{cor: cycleboundarymeasure}~\ref{item:cycleboundary1}, we obtain
    \[
        2^N \mathbb{E}[\xi_{\mathrm{Ver}, q, \Lambda_{n_0}}^\kappa](A_{x_0}) \leq \mathbb{E}[\xi_{\mathrm{Ver}, q, \Lambda_{n_1}}^\kappa](A_{x_0}),
    \]
    so that \(\mu_{n_0}(A_{x_0}) \leq \mu_{n_1}(A_{x_0})\).
    The same argument shows that \(\mu_{n_i}(A_{x_0}) \leq \mu_{n_{i+1}}(A_{x_0})\) for all \(i \geq 0\).       
   
Now, we prove \textit{Claim 1}. Consider the function \(f : [0, \infty) \to [0, \infty)\) defined by
    \[
        f(x) =
        \begin{cases}
            \dfrac{(x - 1)(x - 2) \cdots (x - q - 1)}{(q + 1)!}, &\quad \text{if } x > q + 1, \\
            0, &\quad \text{otherwise}.
        \end{cases}
    \]
    Then, Lemma~\ref{lem:verbose-cardinality} implies that, for any simple marked point set \(X\),
    \begin{equation} \label{eq: number of points via f}
        \xi_{\mathrm{Ver}, q}^\kappa(X)(\Delta') = f(|X|).
    \end{equation}

    It is not difficult to see that \(f\) is a convex function.
    We claim that \(\mathbb{E}[f(\Phi(\Lambda_n))] \geq f(\mathbb{E}[\Phi(\Lambda_n)])\) for each \(n \geq 1\):
    If \(\mathbb{E}[f(\Phi(\Lambda_n))] = \infty\), then the claim holds immediately.  
    Assume \(\mathbb{E}[f(\Phi(\Lambda_n))] < \infty\). Since Equation \eqref{eq:linearity} implies that \(\mathbb{E}[\Phi(\Lambda_n)] < \infty\), the claim holds by Jensen's inequality.
    Equation~\eqref{eq: number of points via f},
    the claim, and Equation \eqref{eq:linearity} imply that, for each \(n \geq 1\),   \[
        \mathbb{E}[\xi_{\mathrm{Ver}, q, \Lambda_n}^\kappa](\Delta') = \mathbb{E}[f(\Phi(\Lambda_n))] 
        \geq f(\mathbb{E}[\Phi(\Lambda_n)]) = f(\lambda n^N).
    \]
    Hence, since \(\lambda > 0\),
    \begin{align*}
        \frac{1}{n^N} \mathbb{E}[\xi_{\mathrm{Ver}, q, \Lambda_n}^\kappa](\Delta') &\geq \frac{1}{n^N} f(\lambda n^N) \\
        &= \lambda^{q+1} (n^N)^q \cdot \frac{f(\lambda n^N)}{(\lambda n^N)^{q+1}} \\
        &\to \infty \quad \text{as } n \to \infty.
    \end{align*}

    Next, we prove \textit{Claim 2.} 
    Observe that \(f(x) > x\) for all \(x > t\), for some sufficiently large constant \(t > 0\).  
    Then, by Equation~\eqref{eq: number of points via f}, for each \(n \geq 1\), we have
    \begin{align*}
        \mathbb{E}[\xi_{\mathrm{Ver}, q, \Lambda_n}^\kappa](\Delta') 
        &= \mathbb{E}[f(\Phi(\Lambda_n))] \\
        &\geq \mathbb{E}[f(\Phi(\Lambda_n)) \mathbbm{1}_{\{\Phi(\Lambda_n) > t\}}] \\
        &\geq \mathbb{E}[\Phi(\Lambda_n) \mathbbm{1}_{\{\Phi(\Lambda_n) > t\}}].
    \end{align*}
    Note that \(\mathbb{E}[\Phi(\Lambda_n)] = \lambda n^N = \infty\).  
    Since
    \begin{align*}
        \mathbb{E}[\Phi(\Lambda_n)] 
        &= \mathbb{E}[\Phi(\Lambda_n) \mathbbm{1}_{\{\Phi(\Lambda_n) > t\}}] + \mathbb{E}[\Phi(\Lambda_n) \mathbbm{1}_{\{\Phi(\Lambda_n) \leq t\}}] \\
        &\leq \mathbb{E}[\Phi(\Lambda_n) \mathbbm{1}_{\{\Phi(\Lambda_n) > t\}}] + \mathbb{E}[t \mathbbm{1}_{\{\Phi(\Lambda_n) \leq t\}}] \\
        &\leq \mathbb{E}[\Phi(\Lambda_n) \mathbbm{1}_{\{\Phi(\Lambda_n) > t\}}] + t,
    \end{align*}
     \(\mathbb{E}[\Phi(\Lambda_n) \mathbbm{1}_{\{\Phi(\Lambda_n) > t\}}]\) is infinite.
    Therefore, \(\mathbb{E}[\xi_{\mathrm{Ver}, q, \Lambda_n}^\kappa](\Delta') = \infty\), and hence
    \[
        \frac{1}{n^N} \mathbb{E}[\xi_{\mathrm{Ver}, q, \Lambda_n}^\kappa](\Delta') = \infty.
    \]
\end{proof}

Next, we prove Proposition~\ref{prop: persbettimeasure}, Theorem~\ref{thm: persbettiSLLNextended}, and Theorem~\ref{thm: persbettiCLTextended}. 

\begin{proof}[Proof of Proposition~\ref{prop: persbettimeasure}]
    By Proposition~\ref{prop: cycleboundary}, we have
    \begin{align*}
        \dim \mathcal{Z}_q(K_r) &= \xi_{\mathrm{Ver}, q}(\mathcal{C}(\K))([0, r] \times [0, \infty]), \\
        \dim \left(\mathcal{Z}_q(K_r) \cap \mathcal{B}_q(K_s)\right) &= \xi_{\mathrm{Ver}, q}(\mathcal{C}(\K))([0, r] \times [0, s]).
    \end{align*}
    Therefore, we have
    \begin{align*}
        \beta_q^{r, s}(\K) 
        &= \dim \frac{\mathcal{Z}_q(K_r)}{\mathcal{Z}_q(K_r) \cap \mathcal{B}_q(K_s)} \\
        &= \dim \mathcal{Z}_q(K_r) - \dim \left(\mathcal{Z}_q(K_r) \cap \mathcal{B}_q(K_s)\right) \\
        &= \xi_{\mathrm{Ver}, q}(\mathcal{C}(\K))([0, r] \times [0, \infty]) 
           - \xi_{\mathrm{Ver}, q}(\mathcal{C}(\K))([0, r] \times [0, s]) \\
        &= \xi_{\mathrm{Ver}, q}(\mathcal{C}(\K))([0, r] \times (s, \infty]).
    \end{align*}
\end{proof}

\begin{proof}[Proof of Theorem~\ref{thm: persbettiSLLNextended}]
    By Theorem~\ref{thm: persbettiSLLN}, it suffices to only consider the case \(r > s\).
    Set \(t := r - s\).
    Then, 
    we have
    \begin{equation} \label{eq: bettishift}
        \begin{aligned}
            \beta_{q, \Lambda_n}^{\kappa, r, s} 
            &= \beta_q^{r, s}(\K^\kappa(\Phi_{\Lambda_n}))&\mbox{by definition} \\
            &= \xi_{\mathrm{Ver}, q}(\mathcal{C}(\K^\kappa(\Phi_{\Lambda_n})))([0, r] \times (s, \infty]) & \text{by Proposition~\ref{prop: persbettimeasure}}\\
            &= \xi_{\mathrm{Ver}, q, \Lambda_n}^\kappa([0, r] \times (s, \infty])&\mbox{by definition} \\
            &= \left(\xi_{\mathrm{Ver}, q, \Lambda_n}^{\kappa^{q, t}}\right)^{((0, t))}([0, r] \times (s, \infty])  & \text{by \Cref{eq: shiftofverbosediagram}}\\
            &= \xi_{\mathrm{Ver}, q, \Lambda_n}^{\kappa^{q, t}}([0, r] \times (r, \infty])&\mbox{by definition} \\
            &= \xi_{\mathrm{Ver}, q}(\mathcal{C}(\K^\kappa(\Phi_{\Lambda_n})))([0, r] \times (r, \infty])&\mbox{by definition}\\
            &= \beta_q^{r, r}(\K^{\kappa^{q, t}}(\Phi_{\Lambda_n})) & \text{by Proposition~\ref{prop: persbettimeasure}}\\
            &= \beta_{q, \Lambda_n}^{\kappa^{q, t}, r, r}&\mbox{by definition}.
        \end{aligned}
    \end{equation}
    By Theorem~\ref{thm: persbettiSLLN} and Remark~\ref{rmk: shirai}, $\beta_{q,\Lambda_n}^{\kappa^{q, t}, r,r}\rightarrow \widehat{\beta}^{\kappa^{q, t}, r,r}_q$ as $n\rightarrow \infty$. This proves  the convergence of $\beta_{q, \Lambda_n}^{\kappa, r, s}$ to \(\widehat{\beta}^{\kappa, r,s}_q := \widehat{\beta}^{\kappa^{q, t}, r,r}_q\) as $n\rightarrow \infty$.
\end{proof}

\begin{proof}[Proof of Theorem~\ref{thm: persbettiCLTextended}] 
    By Theorem~\ref{thm: persbettiCLT}, it suffices to only consider the case \(r > s\).
    Set \(t := r - s\).
    Again by Theorem~\ref{thm: persbettiCLT}, we have
    \[
        \frac{\beta_{q, \Lambda_n}^{\kappa^{q, t}, r, r} - \mathbb{E}[\beta_{q, \Lambda_n}^{\kappa^{q, t}, r, r}]}{n^{N/2}} \xrightarrow{d} \mathcal{N}(0, \sigma_{r, r}^2(\kappa^{q, t}, q))\quad \text{as } n \to \infty.
    \]
    Since $\beta_{q, \Lambda_n}^{\kappa, r, s}=\beta_{q, \Lambda_n}^{\kappa^{q, t}, r, r}$ from the two ends of Equation~\eqref{eq: bettishift}, by setting \(\sigma_{r, s}^2(\kappa, q) := \sigma_{r, r}^2(\kappa^{q, t}, q)\), the above convergence implies the convergence in Equation~\eqref{eq: persbettiCLTextended}.
\end{proof}

Our next goal is to prove Theorem~\ref{thm: measuresupportverbose}. Our proof of Theorem~\ref{thm: measuresupportverbose} follows a line of reasoning analogous to that of Theorem~\ref{thm: measuresupport}, as presented in~\cite{hiraokaetal2018}.
Namely, \cite[Lemma~4.4]{hiraokaetal2018} plays a central role in the proof of Theorem~\ref{thm: measuresupport}, which shows that if \(\kappa\) on \(\mathscr{F}(\R^N)\) is Lipschitz continuous with respect to the Hausdorff distance, then the bottleneck distance between the persistence diagrams of the corresponding \(\kappa\)-filtrations is also Lipschitz continuous with respect to the Hausdorff distance.

The following lemma, proved using Theorem~\ref{thm: isometry}, parallels  \cite[Lemma~4.4]{hiraokaetal2018}.

\begin{lemma} \label{lem: matchingdistancecontinuity}
   Let \(\kappa\) be a filtration function  for unmarked point sets, which is $c_\kappa$-Lipschitz continuous with respect to the Hausdorff distance for some $c_\kappa>0$. 
    Then, for any \(\Xi \in \mathscr{F}(\R^N)\), there exists \(\delta = \delta_\Xi > 0\) such that for any \(q \geq 0\) and any \(\Xi' \in \mathscr{F}(\R^N)\) with \(|\Xi| = |\Xi'|\) and \(d_H(\Xi, \Xi') < \delta\), we have
    \begin{equation}\label{eq:dM leq cdH}
        d_M(D_{\mathrm{Ver}, q}^\kappa(\Xi), D_{\mathrm{Ver}, q}^\kappa(\Xi')) \leq c_\kappa d_H(\Xi, \Xi').
    \end{equation}
\end{lemma}
\begin{proof}
    Let \(\delta := \frac{1}{2} \min \{\|x - x'\| : x, x' \in \Xi,\ x \neq x'\}\).
    Take any \(\Xi' \in \mathscr{F}(\R^N)\) such that \(|\Xi| = |\Xi'|\) and \(d_H(\Xi, \Xi') < \delta\). Then, by the definition of the Hausdorff distance, each point \(x \in \Xi\) has a unique point \(f(x) \in \Xi'\) within distance \(\delta\). Then, it is easy to see that, for any $\sigma\subset \Xi$, 
     \[
        d_H(\sigma, f(\sigma)) \leq d_H(\Xi, \Xi').
    \]
    
    Next, let $t:= c_\kappa d_H(\Xi, \Xi')$. By  Theorem~\ref{thm: isometry}, to prove the inequality in \Cref{eq:dM leq cdH}, 
    it suffices to show that the following FCCs are $t$-interleaved: \[\mathcal{C}(\K^\kappa(\Xi))=(C_*, \partial_C, \ell_C) \mbox{ and } \mathcal{C}(\K^\kappa(\Xi'))= (D_*, \partial_D, \ell_D).\]
    Note that \((C_*, \partial_C)\) and \((D_*, \partial_D)\) are standard simplicial chain complexes of the full simplicial complexes on the vertex sets $\Xi$ and $\Xi'$ respectively. Hence the map \(f\) induces a chain map \(\mathcal{F}_*: C_* \to D_*\) defined as
    \[
        \mathcal{F}_*\left( \sum a_i \sigma_i \right) = \sum a_i f(\sigma_i).
    \]
    
    By Lipschitz continuity of \(\kappa\), for any $\sigma\subset \Xi$, we have
    \begin{align*}
        \kappa(f(\sigma)) 
        &\leq |\kappa(f(\sigma)) - \kappa(\sigma)| + \kappa(\sigma) \\
        &\leq c_\kappa d_H(\sigma, f(\sigma)) + \kappa(\sigma) \\
        &\leq t + \kappa(\sigma),
    \end{align*}
    and hence \(\ell_D \circ \mathcal{F}_* \leq \ell_C + t\).
    Similarly, the inverse \(f^{-1}\) induces a chain map \(\mathcal{G}_*: D_* \to C_*\) with \(\ell_C \circ \mathcal{G}_* \leq \ell_D + t\).
Clearly, \(\mathcal{F}_*\) and \(\mathcal{G}_*\) are inverses of each other, and thus the compositions
    \[
        C_*^\lambda \xrightarrow{\mathcal{F}_*^\lambda} D_*^{\lambda + t} \xrightarrow{\mathcal{G}_*^{\lambda + t}} C_*^{\lambda + 2t}
        \quad \text{and} \quad
        D_*^\lambda \xrightarrow{\mathcal{G}_*^\lambda} C_*^{\lambda + t} \xrightarrow{\mathcal{F}_*^{\lambda + t}} D_*^{\lambda + 2t}
    \]
    are equal to the respective inclusion maps.
    Therefore, \(\mathcal{F}_*\) and \(\mathcal{G}_*\) form a \(t\)-interleaving.
\end{proof}

Next, we introduce the notion of a \emph{marker}, which, roughly speaking, is a finite point set in $\R^N$ that induces
 a specified point in 
a verbose diagram. This notion is a direct extension of  \cite[Definition~4.5]{hiraokaetal2018} in the sense that, by replacing $\Delta'$ and \(\xi_{\mathrm{Ver}, q}^\kappa\) with $\Delta$ and \(\xi_{q}\), respectively, in the definition below, we obtain precisely \cite[Definition~4.5]{hiraokaetal2018}.

\begin{definition}
    Let \(\Lambda\) be a bounded Borel set in \(\R^N\), and let \((b, d) \in \Delta'\).
    For \(q \geq 0\) and a filtration function \(\kappa\) for unmarked point sets, we say that a set \(\Xi \subset \mathscr{F}(\R^N)\)  is a \emph{\((b, d)\)-marker} in \(\Lambda\) (for the degree-\(q\) verbose diagram of \(\kappa\)) if:
    \begin{itemize}
        \item[(i)] \(\Xi \subset \Lambda\), and  
        \item[(ii)] for any \(X \in \mathscr{F}(\R^N)\),
        \[
            \xi_{\mathrm{Ver}, q}^\kappa((X \setminus \Lambda) \sqcup \Xi)(\{(b, d)\}) \geq \xi_{\mathrm{Ver}, q}^\kappa(X \setminus \Lambda)(\{(b, d)\}) + 1.
        \]
    \end{itemize}
    For a subset \(A \subset \Delta'\), we also say that \(\Xi\) is an \(A\)-\emph{marker in \(\Lambda\)} if there exists \((b, d) \in A\) such that \(\Xi\) is a \((b, d)\)-marker in \(\Lambda\).
\end{definition}

\begin{lemma} \label{lem: marker}
    Let \(\kappa: \mathscr{F}(\R^N) \to [0, \infty)\) be a filtration function, which is $c_\kappa$-Lipschitz continuous with respect to the Hausdorff distance for some $c_\kappa>0$.     
    Fix \(A \subset \{(b,d) \in \Delta' : d < \infty\}\), \(q \geq 0\),
    \(\Xi \in \mathscr{F}(\R^N)\), \(\varepsilon > 0\), and a convex averaging sequence \(\{\Lambda_n\}_n\).
    There exists a natural number $M$ satisfying the following property: 
    If \(\Xi' \in \mathscr{F}(\R^N)\) such that \(|\Xi| = |\Xi'|\) and \(d_H(\Xi,\Xi') < \varepsilon\) and
    there is a point in \(A\) which is \((\kappa,q)\)-realizable by \(\Xi'\), i.e.,
    \[
        D_{\mathrm{Ver},q}^\kappa(\Xi') \cap A \neq \varnothing,
    \]
    then, for all $n\geq M$, $\Xi'$ is an
    \(A\)-marker in \(\Lambda_n\) for the degree-\(q\) verbose diagram of \(\kappa\).
\end{lemma}

\begin{proof}
    We use the fact that, for sufficiently large \(n\), \(\Xi'\) lies far from \(\Lambda_n^c\).
    We begin by specifying 
    \(M\).
    By Definition~\ref{def:filtration function} (K3), there exists an increasing function \(\rho: [0, \infty] \to [0, \infty]\) with \(\rho(t) < \infty\) for all finite \(t\), such that
    \(
        \|x-y\| \leq \rho(\kappa(\{x,y\}))
    \)
    for all \(x,y \in \R^N\). Let \(L = \kappa(\Xi) + c_\kappa \varepsilon\).
    Let \(M\) be a sufficiently large natural number so that, for each \(n \geq M\), we have
    \begin{equation} \label{eq: selection of n}
        d(\Xi, \Lambda_n^c) := \inf \{\|x-y\| : x \in \Xi,\, y \in \Lambda_n^c\} > \rho(L) + \varepsilon.
    \end{equation}

    Let \((b,d) \in D_{\mathrm{Ver},q}^\kappa(\Xi') \cap A\) and \(X \in \mathscr{F}(\R^N)\), and fix any \(n \geq M\).
    Since \(\xi_{\mathrm{Ver},q}^\kappa(\Xi')(\{(b,d)\}) \geq 1\), it suffices to show that
    \begin{equation} \label{eq: bdinequality}
        \xi_{\mathrm{Ver},q}^\kappa((X \setminus \Lambda_n) \sqcup \Xi')(\{(b,d)\})= \xi_{\mathrm{Ver},q}^\kappa(X \setminus \Lambda_n)(\{(b,d)\}) + \xi_{\mathrm{Ver},q}^\kappa(\Xi')(\{(b,d)\}).
    \end{equation}
    This equality follows from the two claims below.
    
    \noindent\textit{Claim 1.} \(d < L\).
    
    \noindent\textit{Claim 2.} For any Borel set \(B \subset \{(x, y) \in \Delta': y \leq L\}\),
    \[
        \xi_{\mathrm{Ver},q}^\kappa((X \setminus \Lambda_n) \sqcup \Xi')(B) 
        = \xi_{\mathrm{Ver},q}^\kappa(X \setminus \Lambda_n)(B) 
        + \xi_{\mathrm{Ver},q}^\kappa(\Xi')(B).
    \]
    
    By \textit{Claim 1}, \(\{(b, d)\}\) is a subset of \(\{(x, y) \in \Delta': y \leq L\}\), which is clearly Borel.
    Then \textit{Claim 2} gives the equality given in Equation~\eqref{eq: bdinequality}.

    Now, we prove \textit{Claim 1}.
    Since $\kappa$ is  $c_{\kappa}$-Lipschitz and $d_H(\Xi,\Xi')<\varepsilon$, we have that \(\kappa(\Xi') \leq \kappa(\Xi) + c_\kappa \varepsilon = L\).
    Since \((b, d) \in A\) implies \(d < \infty\), there exists a \((q+1)\)-chain \(y\) in the FCC \(\mathcal{C}(\K^\kappa(\Xi'))\) with \(\ell_{\K^\kappa(\Xi')}(y)=d\).
    Every simplex in \(y\) lies in \(\Xi'\), and hence, by Equation~\eqref{eq: vectfilt},
    \[
        d = \ell_{\K^\kappa(\Xi')}(y) \leq \kappa(\Xi') \leq L.
    \]

    Next, we prove \textit{Claim 2}.
    The statement holds if and only if the following equality holds
    \begin{multline*} 
        \{(b, d) \in D_{\mathrm{Ver}, q}^\kappa((X \setminus \Lambda_n) \sqcup \Xi'): d \leq L\} \\ 
        = \{(b, d) \in D_{\mathrm{Ver}, q}^\kappa(X \setminus \Lambda_n): d \leq L\} \sqcup \{(b, d) \in D_{\mathrm{Ver}, q}^\kappa(\Xi'): d \leq L\}.
    \end{multline*}
    To show this, we exploit another claim that will be proved at the end of the proof:

    \noindent\textit{Claim 3.} For \(t \leq L\),
    the simplicial complex \(K^\kappa((X \setminus \Lambda_n) \sqcup \Xi', t)\)  (cf. Equation~\eqref{eq: sublevel}) is the disjoint union of \(K^\kappa(X \setminus \Lambda_n, t)\) and \(K^\kappa(\Xi', t)\).

    For a filtration \(\K = \{K_t\}_{t \in [0, \infty)}\) of a finite simplicial complex and \(t_0 > 0\), denote by \(\K_{\leq t_0}\) the filtration \(\{K_t'\}_{t \in [0, \infty)}\) given by
    \[
        K_t' =
        \begin{cases}
            K_t, &\quad \text{if } t \leq t_0, \\
            K_{t_0}, &\quad \text{if } t > t_0.
        \end{cases}
    \]
    It is not difficult to show that, as multisets,
    \begin{equation} \label{eq: timelimit}
        \{(b, d) \in D_{\mathrm{Ver}, q}(\mathcal{C}(\K_{\leq t_0})): d < \infty\} = \{(b, d) \in D_{\mathrm{Ver}, q}(\mathcal{C}(\K)): d \leq t_0\}.
    \end{equation}
    Also, by \textit{Claim 3}, we have
    \begin{equation} \label{eq: disjoint union}
        D_{\mathrm{Ver}, q}(\mathcal{C}((\K^\kappa((X \setminus \Lambda_n) \sqcup \Xi'))_{\leq L})) = D_{\mathrm{Ver}, q}(\mathcal{C}((\K^\kappa(X \setminus \Lambda_n))_{\leq L})) \sqcup D_{\mathrm{Ver}, q}(\mathcal{C}((\K^\kappa(\Xi'))_{\leq L})).
    \end{equation}
    Then we have
    \begin{align*}
        &\{(b, d) \in D_{\mathrm{Ver}, q}^\kappa((X \setminus \Lambda_n) \sqcup \Xi'): d \leq L\} \\ 
        &= \{(b, d) \in D_{\mathrm{Ver}, q}(\mathcal{C}((\K^\kappa((X \setminus \Lambda_n) \sqcup \Xi'))_{\leq L})): d < \infty\} & \text{by Equation~\eqref{eq: timelimit}}\\
        &= \{(b, d) \in D_{\mathrm{Ver}, q}(\mathcal{C}((\K^\kappa(X \setminus \Lambda_n))_{\leq L})): d < \infty\} \\
        &\qquad \sqcup \{(b, d) \in D_{\mathrm{Ver}, q}(\mathcal{C}((\K^\kappa(\Xi'))_{\leq L})): d < \infty\}& \text{by Equation~\eqref{eq: disjoint union}}\\
        &= \{(b, d) \in D_{\mathrm{Ver}, q}^\kappa(X \setminus \Lambda_n): d \leq L\} \sqcup \{(b, d) \in D_{\mathrm{Ver}, q}^\kappa(\Xi'): d \leq L\}& \text{by Equation~\eqref{eq: timelimit}}
    \end{align*}
    as desired.

    Finally, we prove \textit{Claim 3}.
    Since \(d_H(\Xi,\Xi') < \varepsilon\), Equation~\eqref{eq: selection of n} implies that
    \begin{equation} \label{eq: far enough}
        d(\Xi', \Lambda_n^c) > \rho(L).
    \end{equation}
    Suppose that
    \(\sigma \subset (X \setminus \Lambda_n) \sqcup \Xi'\)
    intersects both \(X \setminus \Lambda_n\) and \(\Xi'\).
    It suffices to show that
    \(L < \kappa(\sigma)\) because this implies that at time \(t\leq L\), the simplices present in \(\K^\kappa((X \setminus \Lambda_n) \sqcup \Xi')\) are exactly those present either in \(\K^\kappa(X \setminus \Lambda_n)\) or in \(\K^\kappa(\Xi')\).
    
    Let \(x \in \sigma \cap (X \setminus \Lambda_n)\) and \(y \in \sigma \cap \Xi'\).
    Then we have \(\rho(L) < \|x - y\|\) by Equation~\eqref{eq: far enough}.
    By Definition~\ref{def:filtration function} (K3), we have \(\|x - y\| \leq \rho(\kappa(\{x,y\}))\).
    Also, by Definition~\ref{def:filtration function} (K1), we have \(\rho(\kappa(\{x,y\})) \leq \rho(\kappa(\sigma))\).
    In sum, we have \(\rho(L) < \rho(\kappa(\sigma))\) and hence \(L < \kappa(\sigma)\).
\end{proof}

For $p\in\Delta'$ and $\varepsilon> 0$, let \(B(p, \varepsilon)\) denote the open ball of radius \(\varepsilon\) centered at \(p\), with respect to the Euclidean distance. The following lemma, an extension of \cite[Theorem~4.7]{hiraokaetal2018}, will play a crucial role in the proof of Theorem~\ref{thm: measuresupportverbose} below.

\begin{lemma} \label{lem: subsetofsupport}
    Fix a filtration function \(\kappa: \mathscr{F}(\R^N) \to [0, \infty)\) and a point process \(\Phi\) on \(\R^N\). 
    For \(n \geq 1\), let \(\Lambda_n := [-n/2, n/2)^N\).
    For \(p = (b, d) \in \Delta'\) with \(d < \infty\), any integer \(q \geq 0\), and any real number \(\varepsilon > 0\), consider the event
    \[
        A_{q, \varepsilon, p} := \bigcup_{n = 1}^\infty \{\Phi_{\Lambda_n} \text{ is a } B(p, \varepsilon)\text{-marker in } \Lambda_n \text{ for the degree-\(q\) verbose diagram of } \kappa\}.
    \]
    Then, for the set
    \begin{equation} \label{eq: sq}
        S_q := \bigcap_{\varepsilon > 0} \{(b, d) \in \Delta' : d < \infty,\, \mathbb{P}(A_{q, \varepsilon, (b, d)}) > 0\},
    \end{equation}
    we have \(S_q \subset \mathrm{supp}(\nu_q')\).
\end{lemma}

The proof is essentially the same as that of \cite[Theorem~4.7]{hiraokaetal2018}: From their proof, we obtain a proof simply by 
replacing \(\Delta\) with \(\Delta'\), \(\xi_q\) with \(\xi_{\mathrm{Ver}, q}^\kappa\), \(\nu_q\) with \(\nu_q'\), and ``PH\(_q\)" with ``the degree-\(q\) verbose diagram of \(\kappa\)". 

Before we prove Theorem~\ref{thm: measuresupportverbose}, we give a remark on the distribution of a homogeneous Poisson point process.

\begin{remark}[{\cite[p.~25]{hiraokaetal2018}}] \label{rmk: PPPdistribution}
    Let \(\Lambda \subset \R^N\) be a bounded Borel set. 
    By Proposition~\ref{prop: intvaluedirac}, the set \(\mathscr{N}(\Lambda)\) of all integer-valued Radon measures on \(\Lambda\) can be identified with
    \[
        \left(\bigsqcup_{k = 0}^\infty \Lambda^k\right) / \sim,
    \]
    where \(\sim\) denotes the equivalence relation induced by coordinate permutations and \(\Lambda^0 := \{\varnothing\}\). 
    
    Let \(\Psi\) be a homogeneous Poisson point process on \(\R^N\) with intensity \(\lambda > 0\). 
    Under the above identification, sampling a measure with respect to the distribution \(\Theta_{r_\Lambda(\Psi)}\) can be described as follows:
    \begin{enumerate}[(i)]
        \item Sample a Poisson random variable \(n\) with mean \(\lambda \cdot \vol{\Lambda}\);
        \item Conditional on \(n\), sample \(n\) i.i.d.\ points \((x_1,\dots,x_n)\) uniformly from \(\Lambda\), independently of \(n\);
        \item Take the equivalence class of \((x_1,\dots,x_n)\) under \(\sim\).
    \end{enumerate}
    
    For \(k > 0\), the local Janossy density of \(\Theta_{r_\Lambda(\Psi)}\) on \(\Lambda^k\) (see, e.g., \cite[Definition~5.4.IV]{daley2003introduction}) is given by
    \begin{align*}
        \Theta_{r_\Lambda(\Psi)}(dx_1 \dots dx_k)
        &= k! \cdot \mathbb{P}(n = k) \cdot \frac{1}{\vol{\Lambda}^k}\, dx_1 \dots dx_k \\
        &= \lambda^k e^{-\lambda \cdot \vol{\Lambda}} \, dx_1 \dots dx_k.
    \end{align*}
\end{remark}

We finally prove Theorem~\ref{thm: measuresupportverbose}.

\begin{proof}[Proof of Theorem~\ref{thm: measuresupportverbose}]
    By Lemma~\ref{lem: nomark}~\ref{lemitem: verbosesupport}, it suffices to show that the assumption of Theorem~\ref{thm: measuresupport} implies the conclusion of Theorem~\ref{thm: measuresupportverbose}.
    Suppose that the assumption of Theorem~\ref{thm: measuresupport} holds.
    
    Fix an integer \(q \geq 0\).
    We first show that \(\mathrm{supp}(\nu_q') \subset \overline{R_q'}\).
    If \((b, d) \notin \overline{R_q'}\), then there exists a neighborhood \(U\) of \((b, d)\) such that \(\xi_{\mathrm{Ver}, q, \Lambda_n}^\kappa(U) = 0\) for all \(n\).
    By Lemma~\ref{lem: vagueconvequiv} and Theorem~\ref{thm: vagueconv}, this implies that \(\nu_q'(U) = 0\) and thus \((b, d) \notin \mathrm{supp}(\nu_q')\).

    Next, we show that \(\overline{R_q'} \subset \mathrm{supp}(\nu_q')\).
   Since \(\mathrm{supp}(\nu_q')\) is closed, it suffices to prove that \({R_q'} \subset \mathrm{supp}(\nu_q')\).
    Let \((b, d) \in R_q'\).
    Since \(\kappa\) is finite-valued,
    either  
    (i) \(d < \infty\), or  
    (ii) \(d = \infty\) and \(q = 0\):
    Suppose that \(q > 0\). For any \(\Xi \in \mathscr{F}(\R^N)\), consider the simplicial filtration $\K^\kappa(\Xi)=\{K_t\}_{t\geq 0}$. Then there exists $T>0$ such that $K_t$ is the full simplicial complex on $\Xi$ for all $t\geq T$ in which every cycle is a boundary. This implies that every point \((b,d)\) in \(D_{\mathrm{Ver},q}^\kappa(\Xi)\) satisfies $d<\infty$. 

    First, assume (i) $d<\infty$. In order to prove that \((b, d) \in \mathrm{supp}(\nu_q')\),
    by Lemma~\ref{lem: subsetofsupport}, it suffices to show that \((b, d)\in S_q\) in Equation~\eqref{eq: sq}.
    For that, it suffices to show that $\mathbb{P}(A_{q, \varepsilon, (b, d)}) > 0$ for all \(\varepsilon > 0\).
    Fix $\varepsilon>0$ and let \((b, d)\) be $(\kappa,q)$-realizable by a set \(\{x_1, \dots, x_m\}\subset \R^N\).
    By Lemma~\ref{lem: matchingdistancecontinuity}, there exists \(\delta > 0\) such that for any \(\{z_1, \dots, z_m\} \in \mathscr{F}(\R^N)\) with \(d_H(\{x_1, \dots, x_m\}, \{z_1, \dots, z_m\}) < \delta\), we have
        \[d_M(D_{\mathrm{Ver}, q}^\kappa(\{x_1, \dots, x_m\}), D_{\mathrm{Ver}, q}^\kappa(\{z_1, \dots, z_m\}))
        \leq c_\kappa d_H(\{x_1, \dots, x_m\}, \{z_1, \dots, z_m\}).\]
    There exists \(r > 0\), smaller than both \(\varepsilon / c_\kappa\) and \(\delta\), such that the open balls \(B(x_1, r), \dots, B(x_m, r)\) are pairwise disjoint.
    Pick any \(y_i \in B(x_i, r)\) for each \(i\).
    Then the set \(\{y_1, \dots, y_m\}\) is within Hausdorff distance less than \(r\) from \(\{x_1, \dots, x_m\}\), and the points \(y_1, \dots, y_m\) are all distinct.
    Since \(d_H(\{x_1, \dots, x_m\}, \{y_1, \dots, y_m\}) < r < \delta\), we have    
    \[
        \quad d_M(D_{\mathrm{Ver}, q}^\kappa(\{x_1, \dots, x_m\}), D_{\mathrm{Ver}, q}^\kappa(\{y_1, \dots, y_m\})) 
        \leq c_\kappa d_H(\{x_1, \dots, x_m\}, \{y_1, \dots, y_m\}) 
        < c_\kappa r 
        < \varepsilon.
    \]
    This implies that in the $\varepsilon$-ball centered at \( (b,d)\in D_{\mathrm{Ver}, q}^\kappa(\{x_1, \dots, x_m\})\) in $d_{\infty}$, there exists a point in \(D_{\mathrm{Ver}, q}^\kappa(\{y_1, \dots, y_m\})\). That point is in the \(\sqrt{2} \varepsilon\)-ball centered at $(b,d)$ in Euclidean distance.
    
    Then, by Lemma~\ref{lem: marker}, for sufficiently large \(n\), any such set \(\{y_1, \dots, y_m\}\) is a \(B((b, d), \sqrt{2}\varepsilon)\)-marker in \(\Lambda_n = [-n/2, n/2)^N\).
    Let \(f_{\Lambda_n} := {d\Theta_{r_{\Lambda_n}(\Phi)} / d\Theta_{r_{\Lambda_n}(\Psi)}}\).
    Then, by Remark~\ref{rmk: PPPdistribution}, we obtain
    \begin{align*}
       &\mathbb{P}(A_{q,\sqrt{2}\varepsilon,(b,d)}) \\&=
        \mathbb{P}(\Phi_{\Lambda_n} \text{ is a } B((b,d), \sqrt{2}\varepsilon)\text{-marker in } \Lambda_n \text{ for the degree-\(q\) verbose diagram of } \kappa) \\
        &\geq {\Theta_{r_{\Lambda_n}(\Phi)}} \left( 
            \bigcap_{i=1}^m \left\{ \Phi(B(x_i, \delta)) = 1 \right\}
            \cap 
            \left\{ \Phi\left( \Lambda_n \setminus \bigcup_{i=1}^m B(x_i, \delta) \right) = 0 \right\}
        \right) \\
        &= \int_{B(x_1, \delta) \times \cdots \times B(x_m, \delta)} 
            f_{\Lambda_n}(y_1, \ldots, y_m) \, {\Theta_{r_{\Lambda_n}(\Psi)}}(dy_1 \cdots dy_m) \\
        &= \lambda^m e^{-\lambda\cdot\vol{\Lambda_n}} \int_{B(x_1, \delta) \times \cdots \times B(x_m, \delta)} 
            f_{\Lambda_n}(y_1, \ldots, y_m) \, dy_1 \cdots dy_m \\
        &> 0.
    \end{align*}
     Since \(\varepsilon\) was arbitrary, we have proved that $\mathbb{P}(A_{q,\varepsilon,(b,d)})>0$ for all $\varepsilon>0$.
     
    Next, assume (ii) \(d = \infty\) and \(q = 0\). Let \(U\) be any open neighborhood of \((b, \infty)\).
    In order to prove that \((b, \infty) \in \mathrm{supp}(\nu_0')\), it suffices to show that \(\nu_0'(U) > 0\).
    Choose \(t > 0\) such that \(\{b\} \times [t, \infty] \subset U\).
    By Definition~\ref{def:filtration function} (K3), there exists an increasing function \(\rho: [0, \infty] \to [0, \infty]\) with \(\rho(u) < \infty\) for all finite \(u\), such that \(\|x - y\| \leq \rho(\kappa(\{x,y\}))\) for all \(x, y \in \R^N \).
    Consider any set \(\{x, y\} \in \mathscr{F}(\R^N)\) with \( \rho(t)<\|x - y \| \). Since $\rho$ is increasing, this implies that \(t < \kappa(\{x, y\})\).
    Then we have \((b, \kappa(\{x, y\})) \in \{b\} \times [t, \infty] \subset U\).
    Also, we have \((b, \kappa(\{x, y\})) \in D_{\mathrm{Ver}, q}^\kappa(\{x, y\})\), implying that \((b, \kappa(\{x, y\})) \in R_0'\).
    Since $\kappa(\{x, y\})<\infty$, by case (i), it follows that \((b, \kappa(\{x, y\})) \in \mathrm{supp}(\nu_0')\).
    Since \(U\) is a neighborhood of \((b, \kappa(\{x, y\}))\), we have that \(\nu_0'(U) > 0\), as desired.
  \end{proof}

\section{Conclusions}

We extended the main results of \cite{hiraokaetal2018} and \cite{shirai2022} to the setting of verbose diagrams.
Also, we extended the fundamental lemma of persistent homology to the verbose diagram setting.

We suggest the following directions for future research.
\begin{enumerate}
    \item Asymptotic behavior is just one aspect of studying random verbose diagrams. In fact, many works explore properties of random persistence diagrams beyond their asymptotic behavior, some of which may naturally extend to verbose diagrams, e.g.,
    the study of the density of expected persistence diagrams \cite{chazaldivol2018}
    and the study of probability measures on the space of persistence diagrams 
    \cite{mileyko2011, munch2015}.
    
    \item 
    One may consider statistical analysis of verbose diagrams (VDs) ---e.g., hypothesis testing on sets of VDs, statistical inference based on VDs, or machine learning using vectorized representations of VDs, as an extension of such works on persistence diagrams \cite{bubenik2015, adams2017, Fasyetal2014, robinson2017}.
\end{enumerate}

\printbibliography

@article{carlsson2009topology,
  title={Topology and {D}ata},
  author={Carlsson, Gunnar},
  journal={Bulletin of the American Mathematical Society},
  volume={46},
  number={2},
  pages={255--308},
  year={2009}
}

@book{carlsson2021topological,
  title={Topological data analysis with applications},
  author={Carlsson, Gunnar and Vejdemo-Johansson, Mikael},
  year={2021},
  publisher={Cambridge University Press}
}

@book{dey2022computational,
  title={Computational topology for data analysis},
  author={Dey, Tamal Krishna and Wang, Yusu},
  year={2022},
  publisher={Cambridge University Press}
}

@article{hiraoka2016hierarchical,
  title={Hierarchical structures of amorphous solids characterized by persistent homology},
  author={Hiraoka, Yasuaki and Nakamura, Takenobu and Hirata, Akihiko and Escolar, Emerson G and Matsue, Kaname and Nishiura, Yasumasa},
  journal={Proceedings of the National Academy of Sciences},
  volume={113},
  number={26},
  pages={7035--7040},
  year={2016},
  publisher={National Academy of Sciences}
}

@article{hiraokaetal2018,
    author = {Yasuaki Hiraoka and Tomoyuki Shirai and Khanh Duy Trinh},
    title = {{Limit theorems for persistence diagrams}},
    volume = {28},
    journal = {The Annals of Applied Probability},
    number = {5},
    publisher = {Institute of Mathematical Statistics},
    pages = {2740 -- 2780},
    year = {2018},
    doi = {10.1214/17-AAP1371},
    URL = {https://doi.org/10.1214/17-AAP1371}
}

@article{mémolizhou2024,
  title={Ephemeral persistence features and the stability of filtered chain complexes},
  author={Zhou, Ling and M{\'e}moli, Facundo},
  journal={Journal of Computational Geometry},
  volume={15},
  number={2},
  pages={258--328},
  year={2024}
}

@book{hatcher2002algebraic,
  title={Algebraic Topology},
  author={Hatcher, Allen},
  year={2002},
  publisher={Cambridge University Press}
}

@article{usherzhang2016,
  title={Persistent homology and Floer--Novikov theory},
  author={Usher, Michael and Zhang, Jun},
  journal={Geometry \& Topology},
  volume={20},
  number={6},
  pages={3333--3430},
  year={2016},
  publisher={Mathematical Sciences Publishers}
}

@book{kallenberg2017,
  title={Random measures, theory and applications},
  author={Kallenberg, Olav and others},
  volume={1},
  year={2017},
  publisher={Springer}
}

@article{crawley-boevey2015,
    author = {Crawley-Boevey, William},
    title = {Decomposition of pointwise finite-dimensional persistence modules},
    journal = {Journal of Algebra and Its Applications},
    volume = {14},
    number = {05},
    pages = {1550066},
    year = {2015},
    doi = {10.1142/S0219498815500668},
}

@book{edelsbrunner2010,
  title={Computational Topology: An Introduction},
  author={Edelsbrunner, H. and Harer, J.},
  series={Applied Mathematics},
  year={2010},
  publisher={American Mathematical Society}
}

@inproceedings{chazaldivol2018,
  title={The Density of Expected Persistence Diagrams and its Kernel Based Estimation},
  author={Chazal, Fr{\'e}d{\'e}ric and Divol, Vincent},
  booktitle={34th International Symposium on Computational Geometry (SoCG 2018)},
  year={2018},
  organization={Schloss-Dagstuhl-Leibniz Zentrum f{\"u}r Informatik}
}

@article{bubenik2015,
  title={Statistical topological data analysis using persistence landscapes.},
  author={Bubenik, Peter and others},
  journal={J. Mach. Learn. Res.},
  volume={16},
  number={1},
  pages={77--102},
  year={2015}
}

@article{adams2017,
  title={Persistence images: A stable vector representation of persistent homology},
  author={Adams, Henry and Emerson, Tegan and Kirby, Michael and Neville, Rachel and Peterson, Chris and Shipman, Patrick and Chepushtanova, Sofya and Hanson, Eric and Motta, Francis and Ziegelmeier, Lori},
  journal={Journal of Machine Learning Research},
  volume={18},
  number={8},
  pages={1--35},
  year={2017}
}

@article{Fasyetal2014,
  author = {Brittany Terese Fasy and Fabrizio Lecci and Alessandro Rinaldo and Larry Wasserman and Sivaraman Balakrishnan and Aarti Singh},
  title = {{Confidence sets for persistence diagrams}},
  volume = {42},
  journal = {The Annals of Statistics},
  number = {6},
  publisher = {Institute of Mathematical Statistics},
  pages = {2301 -- 2339},
  year = {2014},
  doi = {10.1214/14-AOS1252},
  URL = {https://doi.org/10.1214/14-AOS1252}
}

@article{robinson2017,
  title={Hypothesis testing for topological data analysis},
  author={Robinson, Andrew and Turner, Katharine},
  journal={Journal of Applied and Computational Topology},
  volume={1},
  pages={241--261},
  year={2017},
  publisher={Springer}
}

@article{bobrowski2023,
  title={A universal null-distribution for topological data analysis},
  author={Bobrowski, Omer and Skraba, Primoz},
  journal={Scientific reports},
  volume={13},
  number={1},
  pages={12274},
  year={2023},
  publisher={Nature Publishing Group UK London}
}

@article{turner2014,
  title={Persistent homology transform for modeling shapes and surfaces},
  author={Turner, Katharine and Mukherjee, Sayan and Boyer, Doug M},
  journal={Information and Inference: A Journal of the IMA},
  volume={3},
  number={4},
  pages={310--344},
  year={2014},
  publisher={Oxford University Press}
}

@book{burago2001,
  title={A course in metric geometry},
  author={Burago, Dmitri and Burago, Yuri and Ivanov, Sergei and others},
  volume={33},
  year={2001},
  publisher={American Mathematical Society Providence}
}

@article{kahle2011,
  title={Random geometric complexes},
  author={Kahle, Matthew},
  journal={Discrete \& Computational Geometry},
  volume={45},
  pages={553--573},
  year={2011},
  publisher={Springer}
}

@book{penrose2003,
  title={Random geometric graphs},
  author={Penrose, Mathew},
  volume={5},
  year={2003},
  publisher={OUP Oxford}
}

@incollection{bobrowskietal2022,
  title={Random simplicial complexes: models and phenomena},
  author={Bobrowski, Omer and Krioukov, Dmitri},
  booktitle={Higher-Order Systems},
  pages={59--96},
  year={2022},
  publisher={Springer}
}

@article{yogeshwaran2017,
  title={Random geometric complexes in the thermodynamic regime},
  author={Yogeshwaran, D and Subag, Eliran and Adler, Robert J},
  journal={Probability Theory and Related Fields},
  volume={167},
  pages={107--142},
  year={2017},
  publisher={Springer}
}

@article{owada2017,
  title={Limit theorems for point processes under geometric constraints (and topological crackle)},
  author={Owada, Takashi and Adler, Robert J},
  journal={The Annals of Probability},
  volume={45},
  number={3},
  year={2017},
  publisher={Institute of Mathematical Statistics}
}

@article{bobrowski2022,
  title={Homological connectivity in random {\v{C}}ech complexes},
  author={Bobrowski, Omer},
  journal={Probability Theory and Related Fields},
  volume={183},
  number={3},
  pages={715--788},
  year={2022},
  publisher={Springer}
}

@article{bobrowski2017,
  title={On the vanishing of homology in random {\v{C}}ech complexes},
  author={Bobrowski, Omer and Weinberger, Shmuel},
  journal={Random Structures \& Algorithms},
  volume={51},
  number={1},
  pages={14--51},
  year={2017},
  publisher={Wiley Online Library}
}

@article{chacholski2021invariants,
  title={Invariants for Tame Parametrised chain complexes},
  author={Chach{\'o}lski, Wojciech and Giunti, Barbara and Landi, Claudia},
  journal={Homology, Homotopy \& Applications},
  volume={23},
  number={2},
  year={2021}
}

@article{chacholski2023decomposing,
  title={Decomposing filtered chain complexes: Geometry behind barcoding algorithms},
  author={Chach{\'o}lski, Wojciech and Giunti, Barbara and Jin, Alvin and Landi, Claudia},
  journal={Computational Geometry},
  volume={109},
  pages={101938},
  year={2023},
  publisher={Elsevier}
}

@article{botnan2024,
  title={On the consistency and asymptotic normality of multiparameter persistent Betti numbers},
  author={Botnan, Magnus B and Hirsch, Christian},
  journal={Journal of Applied and Computational Topology},
  volume={8},
  number={6},
  pages={1465--1502},
  year={2024},
  publisher={Springer}
}

@article{krebs2025,
  title={On the asymptotic normality of persistent Betti numbers},
  author={Krebs, Johannes and Polonik, Wolfgang},
  journal={Advances in Applied Probability},
  volume={57},
  number={2},
  pages={492--523},
  year={2025},
  publisher={Cambridge University Press}
}

@article{owada2020,
  title={Convergence of persistence diagrams for topological crackle},
  author={Owada, Takashi and Bobrowski, Omer},
  journal={Bernoulli},
  volume={26},
  number={3},
  year={2020}
}

@article{krebs2021,
  title={On limit theorems for persistent Betti numbers from dependent data},
  author={Krebs, Johannes},
  journal={Stochastic Processes and their Applications},
  volume={139},
  pages={139--174},
  year={2021},
  publisher={Elsevier}
}

@article{shirai2022,
  title={A limit theorem for persistence diagrams of random filtered complexes built over marked point processes},
  author={Shirai, Tomoyuki and Suzaki, Kiyotaka},
  journal={Modern Stochastics: Theory and Applications},
  volume={10},
  number={1},
  pages={1--18},
  year={2022},
  publisher={VTeX: Solutions for Science Publishing}
}

@article{owada2022,
  title={Convergence of persistence diagram in the sparse regime},
  author={Owada, Takashi},
  journal={The Annals of Applied Probability},
  volume={32},
  number={6},
  pages={4706--4736},
  year={2022},
  publisher={Institute of Mathematical Statistics}
}

@article{fasy2024,
  title={How Small Can Faithful Sets Be? Ordering Topological Descriptors},
  author={Fasy, Brittany Terese and Millman, David L and Schenfisch, Anna},
  journal={arXiv preprint arXiv:2402.13632},
  year={2024}
}

@article{belton2020,
  title={Reconstructing embedded graphs from persistence diagrams},
  author={Belton, Robin Lynne and Fasy, Brittany Terese and Mertz, Rostik and Micka, Samuel and Millman, David L and Salinas, Daniel and Schenfisch, Anna and Schupbach, Jordan and Williams, Lucia},
  journal={Computational Geometry},
  volume={90},
  pages={101658},
  year={2020},
  publisher={Elsevier}
}

@article{mileyko2011,
  title={Probability measures on the space of persistence diagrams},
  author={Mileyko, Yuriy and Mukherjee, Sayan and Harer, John},
  journal={Inverse Problems},
  volume={27},
  number={12},
  pages={124007},
  year={2011},
  publisher={IOP Publishing}
}

@article{munch2015,
  title={Probabilistic Fr{\'e}chet means for time varying persistence diagrams},
  author={Munch, Elizabeth and Turner, Katharine and Bendich, Paul and Mukherjee, Sayan and Mattingly, Jonathan and Harer, John},
  journal={Electronic Journal of Statistics},
  volume={9},
  pages={1173--1204},
  year={2015}
}

@article{bobrowski2024universality,
  title={Universality in random persistent homology and scale-invariant functionals},
  author={Bobrowski, Omer and Skraba, Primoz},
  journal={arXiv preprint arXiv:2406.05553},
  year={2024}
}

@article{bubenik2020persistent,
  title={Persistent homology detects curvature},
  author={Bubenik, Peter and Hull, Michael and Patel, Dhruv and Whittle, Benjamin},
  journal={Inverse Problems},
  volume={36},
  number={2},
  pages={025008},
  year={2020},
  publisher={IOP Publishing}
}

@book{kechris2012classical,
  title={Classical descriptive set theory},
  author={Kechris, Alexander},
  volume={156},
  year={2012},
  publisher={Springer Science \& Business Media}
}

@book{last2018lectures,
  title={Lectures on the Poisson process},
  author={Last, G{\"u}nter and Penrose, Mathew},
  volume={7},
  year={2018},
  publisher={Cambridge University Press}
}

@book{daley2003introduction,
  title={An introduction to the theory of point processes: volume I: elementary theory and methods},
  author={Daley, Daryl J and Vere-Jones, David},
  year={2003},
  publisher={Springer}
}

@book{durrett2019probability,
  title={Probability: theory and examples},
  author={Durrett, Rick},
  volume={49},
  year={2019},
  publisher={Cambridge university press}
}

@book{cohn2013measure,
  title={Measure theory},
  author={Cohn, Donald L},
  volume={1},
  year={2013},
  publisher={Springer}
}

@article{kanazawa2024large,
  title={Large deviation principle for persistence diagrams of random cubical filtrations},
  author={Kanazawa, Shu and Hiraoka, Yasuaki and Miyanaga, Jun and Tsunoda, Kenkichi},
  journal={Journal of Applied and Computational Topology},
  volume={8},
  number={6},
  pages={1649--1700},
  year={2024},
  publisher={Springer}
}

\appendix

\section{Proof of Proposition~\ref{prop: intvaluedirac}} \label{apdx: pfofintvaluedirac}

\begin{proof}
    By the atomic decomposition \cite[Lemma~1.6]{kallenberg2017} of \(\mu\),
    there exist a finite or countably infinite indexing set \(I\), positive numbers \((\beta_i)_{i \in I}\), and distinct points \((\sigma_i)_{i \in I}\) in \(S\), along with a \emph{diffuse} measure \(\alpha\) on \(S\)---that is, \(\alpha(\{p\}) = 0\) for any \(p \in S\)---such that
    \[
        \mu = \alpha + \sum_{i \in I} \beta_i \delta_{\sigma_i}.
    \]
    Each \(\beta_i\) must be an integer, since \(\mu(\{\sigma_i\}) = \beta_i\).
    Thus, it remains to prove that \(\alpha = 0\).
    It therefore suffices to show that any integer-valued diffuse Radon measure must be zero.

    For contradiction, suppose that \(\alpha\) is a nonzero integer-valued diffuse Radon measure.
    Since \(S\) is Polish, it admits a metric \(d\) that generates its topology, and there exists a countable dense sequence \(\{p_n\}_n\) in \(S\).
    Consider the sequence of open balls \(\{B(p_n, 1/2)\}_n\) of radius \(1/2\) centered at the points \(p_n\).
    As \(\bigcup_n B(p_n, 1/2) = S\) and \(\alpha(S) > 0\), there exists some index \(i_1\) such that \(\alpha(B(p_{i_1}, 1/2)) > 0\).
    Set \(q_1 := p_{i_1}\).
    
    We now define \(q_2, q_3, \dots\) inductively.
    Suppose that \(q_1, \dots, q_m\) have been chosen.
    For each \(m \geq 1\), observe that the set of points in \(\{p_n\}_n\) contained in \(B(q_m, 2^{-m})\) is dense in \(B(q_m, 2^{-m})\).
    If \(\alpha(B(q_m, 2^{-m})) > 0\), then there exists some index \(i_{m+1}\) such that \(p_{i_{m+1}} \in B(q_m, 2^{-m})\) and \(\alpha(B(p_{i_{m+1}}, 2^{-(m+1)})) > 0\).
    We then define \(q_{m+1} := p_{i_{m+1}}\) and proceed inductively.
    This construction ensures that \(\alpha(B(q_m, 2^{-m})) > 0\) for all \(m \geq 1\).
    
    Note that \(d(q_i, q_{i+1}) < 2^{-i}\) for each \(i\), so the sequence \(\{q_n\}_{n \geq 1}\) is Cauchy.
    Since \(S\) is complete, the sequence converges to some point \(q \in S\).
    For each \(n \geq 1\), define \(V_n := B(q, 2^{-n})\).
    Then, for each \(n\), there exists \(l > n\) such that \(d(q, q_l) < 2^{-(n+1)}\).
    Because \(l \geq n+1\), we have \(B(q_l, 2^{-l}) \subset V_n\), and hence \(\alpha(V_n) > 0\).
    Moreover, since \(\alpha\) is integer-valued, it follows that \(\alpha(V_n) \geq 1\) for all \(n\).
    Now, observe that \(\bigcap_{n \geq 1} V_n = \{q\}\), and by the continuity from above of measures, we obtain \(\alpha(\{q\}) = \lim_{n \to \infty} \alpha(V_n) \geq 1\), contradicting the assumption that \(\alpha\) is diffuse.
\end{proof}

\end{document}